\documentclass[11pt]{amsart}

\usepackage{amsmath}
\usepackage{amsthm}
\usepackage{amsfonts}
\usepackage{comment}
\usepackage{amssymb,amsrefs}
\usepackage{hyperref}
\usepackage{mathtools}
\usepackage{bm}
\usepackage[margin=1in]{geometry}
\usepackage{mathrsfs}
\usepackage{enumerate}
\usepackage{enumitem}
\usepackage{todonotes}

\newtheorem{theorem}{Theorem}[section]

\newtheorem{lemma}[theorem]{Lemma}
\newtheorem{corollary}[theorem]{Corollary}
\newtheorem{conjecture}[theorem]{Conjecture}

\theoremstyle{definition}
\newtheorem{definition}[theorem]{Definition}

\newtheorem*{theorem*}{Theorem}
\newtheorem*{proposition*}{Proposition}
\newtheorem*{lemma*}{Lemma}

\theoremstyle{remark}

\newtheorem*{note}{Note}

\numberwithin{equation}{section}

\def\al{{\alpha}}
\def\be{{\beta}}

\def\ld{{\lambda}}

\def\CC{{\mathbb C}}

\def\HH{{\mathbb H}}

\def\Ak1{{\mathbb A_k ^1}}

\def\RR{{\mathbb R}}

\def\ZZ{{\mathbb Z}}
\makeatletter
\newcommand*{\defeq}{\mathrel{\rlap{%
                     \raisebox{0.3ex}{$\m@th\cdot$}}%
                     \raisebox{-0.3ex}{$\m@th\cdot$}}%
                     =}
\makeatother

\newcommand{\Conv}{\mathop{\scalebox{1.5}{\raisebox{-0.2ex}{$\ast$}}}}%
\def\cal{\mathcal}
\def\cA{{\cal A}}
\def\cB{{\cal B}}
\def\cC{{\cal C}}
\def\cD{{\cal D}}

\def\cG{{\cal G}}

\def\gone{{G_1 \setminus \{1\}}}
\def\gtwo{{G_2 \setminus \{1\}}}
\def\goc{{\left(\left|G_1\right|-1\right)}}
\def\gtc{{\left(\left|G_2\right|-1\right)}}

\def\cS{{\mathfrak S}}

\def\ii{{\mathfrak i}}

\def\cos{\operatorname{cos}}

\def\Im{\operatorname{Im}}

\def\SL{\operatorname{SL}}

\def\Hom{\operatorname{Hom}}

\def\Tr{\operatorname{Tr}}

\def\length{\operatorname{length}}

\def\PSL{\operatorname{PSL}}
\def\conj{\operatorname{conj}}

\newcommand\floor[1]{\left\lfloor#1\right\rfloor}

\newcommand*{\shom}{\mathscr{H}\kern -.5pt om}
\allowdisplaybreaks
\begin{document}

\title[Conjugacy Growth of Commutators]{Conjugacy Growth of Commutators}

\date{\today}

\author[Peter S. Park]{Peter S. Park}
\address{Department of Mathematics, Harvard University, Cambridge, MA 02138}
\email{\href{mailto:pspark@math.harvard.edu}{{\tt pspark@math.harvard.edu}}}

\begin{abstract}
For the free group $F_r$ on $r>1$ generators (respectively, the free product  $G_1 * G_2$ of two nontrivial finite groups $G_1$ and $G_2$), we obtain the asymptotic for the number of conjugacy classes of commutators in $F_r$ (respectively, $G_1 * G_2$) with a given word length in a fixed set of free generators (respectively, the set of generators given by the nontrivial elements of $G_1$ and $G_2$). Our result is proven by using the classification of commutators in free groups and in free products by Wicks, and builds on the works of Rivin and Sharp, who asymptotically counted the conjugacy classes of commutator-subgroup elements in $F_r$ with a given word length. 
\end{abstract}
\maketitle

\section{Introduction}\label{sec:intro}
Let $G$ be a finitely generated group with a finite symmetric set of generators $\cS$. Any element $g\in G$ can then be written as a word in the letters of $\cS$, and one can  define the \emph{length} of $g$ by
\[
\inf\left\{  k : \exists c_{1},\ldots,c_{k} \text{ s.t. }g=c_{1}\ldots c_{k}\right\}.
\]
Consider the closed ball $B_k(G,\cS) \subset G$ of radius $k$ in the word metric, defined as the subset consisting of elements with length $\le k$. One can then ask natural questions about the growth of $G$: how large is $|B_k(G,\cS)|$ as $k \to \infty$, and more generally, what connections can be made between the properties of $G$ and this notion of its growth rate? For example, one of the pioneering results on this question is that of Gromov~\cite{gromov}, who classified the groups $G$ with polynomial growth.

Since the middle of the 20th century, the growth of groups has been widely studied in various contexts largely arising from geometric motivations, such as characterizing the volume growth of Riemannian manifolds and Lie groups. In fact, in addition to applying knowledge on the growth of groups to describe growth in such geometric settings, one can also pass information in the other direction, i.e., use information on volume growth in geometric settings to yield consequences regarding the growth of the relevant groups. For instance, Milnor~\cite{milnor} used inequalities relating the volume and curvature of Riemannian manifolds to prove that the fundamental group of any compact, negatively-curved Riemannian manifold has at least exponential growth. 

A group that arises especially commonly in such geometric contexts is the free group $F_r$ on $r>1$ generators, which also has exponential growth; more precisely, after fixing a symmetric generating set $\cS\defeq\{x_1,\ldots,x_r,x_1^{-1}, \ldots, x_r^{-1}\}$, it is easy to see that 
\begin{align*}
|B_k(G,\cS)| = 1+\sum_{i=1}^{k} |\partial B_i(G,\cS)| &= 1+ \sum_{i=1}^{k} 2r(2r-1)^{i-1} \\&= 1+\frac{r\left((2r-1)^k-1\right)}{r-1},
\end{align*}
where $\partial B_k(G,\cS)$ denotes the subset of length-$k$ elements.

\subsection{Motivation and main results}
In certain contexts, it is more natural to consider the growth rate of the conjugacy classes of $G$. For a given conjugacy class $\cC$ of $G$, define the \emph{length} of $\cC$ by
\[
\inf_{g \in \cC} \length(g),
\]
and define $\partial B_k^{\conj}(G,\cS)$ as the set of conjugacy classes of $G$ with length $k$.  In the case of $F_r$, the minimal-length elements of a conjugacy class are precisely its cyclically reduced elements, all of which are cyclic conjugates of each other. The conjugacy growth of $F_r$ can be described as $\partial B_k^{\conj}(G,\cS) \sim (2r-1)^k/k$, which agrees with the intuition of identifying the cyclic conjugates among the $2r(2r-1)^{k-1}$ words of length $k$; for the full explicit formula, see~\cite[Proposition 17.8]{conj}. 

One context for which conjugacy growth may be a more natural quantity to study than the growth rate in terms of elements is when characterizing the frequency with which a conjugacy-invariant property occurs in $G$. An example of such a property is membership in the commutator subgroup $[G,G]$. On this front, Rivin~\cite{rivin} computed the number $c_k$ of length-$k$ cyclically reduced words in $F_r$ that are in the commutator subgroup (i.e., have trivial abelianization) to be the constant term in the expression
\[
(2\sqrt{2r-1})^k T_k\left( \frac{1}{2\sqrt{2r-1}} \sum_{i=1}^r \left(x_i+ \frac{1}{x_i} \right)\right),
\]
where $T_k$ denotes the $k$th Chebyshev polynomial of the first kind. Furthermore, Sharp~\cite{sharp} obtained the asymptotics of $c_k$, given by
\[
c_{2m} \sim \frac{4r(2r-1)^{2m-1}}{(2\pi)^{\frac{r}{2}} \sigma^r m^{\frac{r}{2}}},
\]
where 
\[
\sigma\defeq\left(\frac{1}{\sqrt{2r-1}} \left( 1+\left(\frac{r+\sqrt{2r-1}}{r-\sqrt{2r-1}}\right)^{\frac{1}{2}} \right)\right)^\frac{1}{2};
\]
note also that when $k$ is odd, we have $c_k=0$. From the number of cyclically reduced words with trivial abelianization, one can derive the  growth of  conjugacy classes with trivial abelianization by using M\"{o}bius inversion, due to the following relationships:
\[
c_k = \sum_{d \mid k} p_d,
\]
where $p_d$ denotes the number of primitive (i.e., not a proper power of any subword) length-$d$ words with trivial abelianization, and 
\[
 |\partial B_k^{\conj}(G,\cS) \cap [G,G]| = \sum_{d \mid k} \frac{p_d}{d},
\]
which together imply by M\"{o}bius inversion that
\begin{align}
 |\partial B_k^{\conj}(G,\cS) \cap [G,G]| = \sum_{d \mid k} \frac{1}{d}\sum_{e \mid d} \mu\hspace{-1.5pt}\left( \frac{d}{e}\right) c_e &= \sum_{e \mid k} \frac{c_e}{e} \left( \sum_{d' \mid \frac{k}{e}} \frac{\mu(d')}{d'}\right)\nonumber \\ & = \sum_{e \mid k} \frac{c_e}{e} \cdot \frac{\phi(k/e)}{k/e} =  \sum_{e \mid k} c_e\frac{\phi(k/e)}{k}.\label{mobius}
\end{align}
In the above, $\phi$ denotes the Euler totient function. For details on this derivation, the reader is directed to~\cite[Chapter 17]{conj} .

In this paper, we answer the analogous question for commutators rather than for commutator-subgroup elements. This new inquiry is structurally different in that it aims to solve a Diophantine equation over a group $G$ (whether, for a given $W \in G$, there exist $X,Y\in G$ such that $XYX^{-1}Y^{-1}=W$), rather than a subgroup-membership problem (whether $W$ is in $[G,G]$). In particular, the set of commutators is not multiplicatively closed, so we cannot use primitive words as a bridge between counting cyclically reduced words and counting conjugacy classes as above. Instead, we use a theorem of Wicks~\cite{wicks}, which states that an element of $F_r$ is a commutator if and only if it is a cyclically reduced conjugate of a commutator satisfying the following definition.
\begin{definition}
A \emph{Wicks commutator of $F_r$} is a cyclically reduced word  $W\in F_r$ of the form $ABCA^{-1}B^{-1}C^{-1}$. Equivalently, it is a word of the form $ABCA^{-1}B^{-1}C^{-1}$ such that the subwords $A,B,$ and $C$ are reduced; there are no cancellations between the subwords $A,B,C,A^{-1},B^{-1},$ and $C^{-1}$; and the first and last letters are not inverses.
\end{definition}
Using this classification, we count the number of conjugacy classes of commutators in $F_r$ with length $k$  by counting the number of Wicks commutators with length $k$.

\begin{theorem}\label{thm1}
Let $k \ge 0$ be even. The number of distinct conjugacy classes of commutators  in $F_r$ with length $k$ is given by
\[
\frac{(2r-2)^2(2r-1)^{\frac{k}{2}-1}}{96r} \left(k^2 + O_r\left(k\right)  \right),
\]
where the implied constant  depends only on $r$ and is effectively computable.
\end{theorem}
\noindent Note that the number  of conjugacy classes of commutators in $F_r$ is roughly proportional to the square root of the number (\ref{mobius}) of all conjugacy classes with trivial abelianization.

We also employ a similar argument, using Wicks' characterization of commutators in free products, to answer the analogous question for the free product $G_1 * G_2$ of two nontrivial finite groups. Like free groups, free products have exponential growth, which has been studied recently by Mann~\cite{freeproductsmann}. \iffalse This group is of independent interest as the isomorphism class of $\PSL_2(\ZZ)$; specifically, for the usual generators
\begin{equation}
S\defeq \begin{pmatrix}\label{matrixgenerators}
0 & -1 \\
1 & 0
\end{pmatrix} \hspace{24pt} \text{ and } \hspace{24pt} T \defeq \begin{pmatrix} 
1 & 1 \\
0 & 1 
\end{pmatrix}
\end{equation}
of $\PSL_2(\ZZ)$, we have that $S$ corresponds to a generator of the $G_2$ factor and $ST$, to a generator of the $G_1$ factor. \fi We consider the set of generators  $\cS\defeq  (G_1 \setminus \{1\}) \cup  ( G_2 \setminus \{1\} )$ of $G_1 * G_2$. Then, a theorem of Wicks~\cite{wicks} analogous to the previous one implies that an element of $G_1 * G_2$ is  a commutator if and only if it is a cyclic conjugate of a commutator satisfying the following definition.

\begin{definition}\label{wicksdef2}
A \emph{Wicks commutator of $G_1 *G_2$} is a word $W\in G_1 * G_2$ that is \emph{fully cyclically reduced}, which is to say that the adjacent letters (i.e., nonidentity elements of $G_1$ and $G_2$ in the word) are in different factors of the free product, as are the first and last letters; and in one of the following forms:
\begin{enumerate}
\item $\al \in G_1$ that is a commutator of $G_1$,
\item $\al \in G_2$ that is a commutator of $G_2$,
\item $A \al_1 A \al_2^{-1}$ for $\al_1, \al_2 \in G_2$ that are conjugates,
\item $\al_1  A\al_2^{-1} A $ for $\al_1, \al_2 \in G_1$ that are conjugates,
\item $ABA^{-1}B^{-1}$,
\item $A\al_1 B \al_2  A^{-1} \al_3  B^{-1}\al_4  $ for $\al_1, \al_2, \al_3, \al_4 \in G_2$ satisfying $\al_4\al_3\al_2\al_1=1$,
\item $\al_1 A \al_2  B \al_3  A^{-1}\al_4  B^{-1}$ for $\al_1, \al_2, \al_3, \al_4 \in G_1$ satisfying $\al_4\al_3\al_2\al_1=1$,
\item $A\al_1 B \be_1 C\al_2  A^{-1} \be_2  B^{-1}\al_3 C^{-1}\be_3 $ for $\al_1, \al_2, \al_3, \be_1, \be_2, \be_3\in G_2$ satisfying $\al_3\al_2\al_1=1$ and $\be_3\be_2\be_1=1$, with $A, B,$ and $C$ nontrivial,
\item $\be_1 A\al_1 B \be_2 C\al_2  A^{-1} \be_3  B^{-1}\al_3 C^{-1} $ for $\al_1, \al_2, \al_3 ,\be_1, \be_2, \be_3\in G_1$ satisfying $\al_3\al_2\al_1=1$ and $\be_3\be_2\be_1=1$, with $A, B,$ and $C$ nontrivial.
\end{enumerate}
\end{definition}

A fully cyclically reduced element $W\in G_1 *G_2$ with length $k>1$  alternates between $k/2$ letters in $\gone$ and $k/2$ letters in $\gtwo$, where $k/2$ is necessarily an integer. Thus, if $k>1$ is odd, then there are no fully cyclically reduced elements of $G_1 * G_2$ with length $k$. Furthermore, for $W$ to be a Wicks commutator not of the form $(1)$ or $(2)$ in  Definition~\ref{wicksdef2},  it is necessary not only that $k$ is even, but also that $k/2$ is an even integer. Thus, all  but possibly finitely many (the length-$1$ exceptions) Wicks commutators of $G_1*G_2$ have length divisible by $4$. For any $k$ divisible by $4$, we obtain the number of length-$k$ conjugacy classes in $G_1 * G_2$ comprised of commutators.
\begin{theorem}\label{thm2}
Let $k \ge 0$ be a multiple of $4$. The number of distinct conjugacy classes of commutators in $G_1 * G_2$ with length $k$ is given by
\begin{align*}
 \frac{1}{192}\left(\goc\left( \left|G_2\right|-2 \right)^2 +\left( \left|G_1\right|-2 \right)^2 \gtc \right) & k^2  \goc^{\frac{k}{4}-1} \gtc^{\frac{k}{4}-1} 
 \\&+O_{|G_1|,|G_2|}\left( k\goc^{\frac{k}{4}} \gtc^{\frac{k}{4}}  \right),
\end{align*}
where the implied constant only depends on $|G_1|$ and $|G_2|$, and is effectively computable.
\end{theorem}

Suppose that $4 \mid k> 0$. Then, the set of cyclically reduced elements of $G_1 * G_2$ (whose letters of odd position are elements of $G_1$, without loss of generality) with length $k$ and trivial abelianization bijectively maps to the Cartesian product of two sets: the set of closed paths of length $k/2$ on the complete graph $K_{|G_1|}$ with fixed basepoint $P_1$, and the set of closed paths of length $k/2$ on the complete graph $K_{|G_2|}$ with fixed basepoint $P_2$. For a graph $E$, the number of closed paths on $E$ with length $n$ is given by  $\sum \lambda^n$, where $\lambda$ ranges over the eigenvalues (counted with multiplicity) of the adjacency matrix of $E$. For $m\ge 2$, the adjacency matrix of $K_m$ is the $m\times m$ matrix with diagonal entries $0$ and all other entries $1$; its eigenvalues are $m-1$ (with multiplicity $1$) and $-1$ (with multiplicity $m-1$). Consequently, for even integers $n>0$, the number of closed paths on $K_m$ with length $n$ is given by $(m-1)^n+(m-1)(-1)^n=m(m-1)^{n-1}$. By symmetry between the $m$ vertices, the number of such closed paths with a given basepoint is $(m-1)^{n-1}$. \iffalse
\[
\frac{(m-1)^n+(m-1)(-1)^n}{m},
\]
which is equal to $1$ for $m=2$ and can be written as 
\[
\frac{(m-1)^n}{m} +O(1)
\]
for $m>2$.\fi Thus, the number of cyclically reduced words with length $k$ and trivial abelianization in  $G_1 * G_2$ is given by
\[
(|G_1|-1)^{\frac{k}{2}-1}(|G_2|-1)^{\frac{k}{2}-1}.
\]
\iffalse
\[
 \left\{\begin{array}{lr}
        \displaystyle\frac{(|G_1|-1)^n}{|G_1|} +O(1), & \text{if } |G_1|>2\\
        1, & \text{if } |G_1|=2
        \end{array}\right\} \cdot  \left\{\begin{array}{lr}
        \displaystyle\frac{(|G_2|-1)^n}{|G_2|} +O(1), & \text{if } |G_2|>2\\
        1, & \text{if } |G_2|=2
        \end{array}\right\},
        \]\fi
Applying M\"{o}bius inversion as done in (\ref{mobius}), we see that the number of conjugacy classes of commutators in $G_1 * G_2$ with length $k$ is roughly comparable to the square root of the number of all length-$k$ conjugacy classes with trivial abelianization.

\subsection{Geometric applications}
Counting conjugacy classes of commutators has a topological application. Let $X$ be a connected CW-complex with fundamental group $\pi_1(X)$, and let $\cC$ be a conjugacy class of $\pi_1(X)$ with trivial abelianization, corresponding to the free homotopy class of a homologically trivial loop $\gamma : S^1  \to X$. Then, the \emph{commutator length} of $\cC$, defined as the minimum number of commutators whose product is equal to an element of $\cC$, is also the minimum genus of an orientable surface (with one boundary component) that continuously maps to $X$ so that the boundary of the surface maps to $\gamma$~\cite[Section 2.1]{calegari}. Thus, a conjugacy class of $\pi_1(X)$ is comprised of commutators if and only if its corresponding free homotopy class $\gamma: S^1 \to X$
satisfies the following topological property:
\[
\text{\textit{There exists a torus $Y$  with one boundary component }}\tag{$\star$} 
\]

\vspace{-10pt}

\begin{center}
\textit{and a continuous map $f : Y \to X$ satisfying $f(\partial Y)=\Im \gamma$.}
\end{center}

\noindent Consequently, we obtain the following. 
\begin{corollary}\label{cor1}
Let $X$ be a connected CW-complex.
\begin{enumerate}
\item Suppose $X$ has fundamental group $F_r$ with a symmetric set of free generators $\cS$. Then, the number of free homotopy classes of   loops $\gamma : S^1 \to X$ with length $k$ (in the generators of $\cS$) satisfying Property ($\star$) is given by
\[
\frac{(2r-2)^2(2r-1)^{\frac{k}{2}-1}}{96r} \left(k^2 + O_r\left(k \right)  \right),
\] 
where the implied constant  depends only on $r$ and is effectively computable.

\item  Suppose $X$ has fundamental group $G_1 * G_2$ with the set of generators \[
\cS\defeq (G_1 \setminus \{1\}) \cup (G_2 \setminus \{1\}).
\] 
 Then, the number of free homotopy classes of   loops $\gamma : S^1 \to X$ with length $k$ (in the generators of $\cS$) satisfying Property ($\star$)  is given by
\begin{align*}
\frac{1}{192}\left(\goc\left( \left|G_2\right|-2 \right)^2 +\left( \left|G_1\right|-2 \right)^2 \gtc \right)
& k^2  \goc^{\frac{k}{4}-1} \gtc^{\frac{k}{4}-1} \\  + O_{|G_1|,|G_2|}&\left( k\goc^{\frac{k}{4}} \gtc^{\frac{k}{4}}  \right),
\end{align*}
where the implied constant depends only on $|G_1|$ and $|G_2|$, and is effectively computable. 
\end{enumerate}
\end{corollary}
 
For a connected CW-complex with a natural geometric definition of length for loops, one expects the word length of a free homotopy class of loops to correlate with the geometric length of the free homotopy class, taken to be the infimum of the geometric lengths of its loops, in some way. A trivial example of this phenomenon is in the case of a wedge $\bigvee_{j=1}^r S^1$ of $r$ unit circles, for which there is perfect correlation; a free homotopy class in this space has word length $k$ if and only if the minimal geometric length of a loop in the class is $k$. 

A more interesting example of the connection between word length and geometric length arises from hyperbolic geometry. Let $X\cong \Gamma \backslash \HH$ be a hyperbolic orbifold, where $\Gamma \subset \PSL_2(\RR)$ denotes a Fuchsian group. Then, every free homotopy class of loops in $X$ that does not wrap around a cusp has a unique geodesic representative, so the geometric length of the class is realized as the geometric length of this unique closed geodesic~\cite[Theorem 1.6.6]{buser}. In particular, we have a canonical correspondence between hyperbolic free homotopy classes of loops (i.e., ones that do not wrap around a cusp) and closed geodesics. Let us specialize to the case that $X$ is a pair of pants (surface obtained by removing three disjoint open disks from a sphere). We have that $X\cong\Gamma \backslash \HH$ for $\Gamma \cong F_2$, so fix a symmetric set of free generators. Then, Chas, Li, and Maskit~\cite{chas} have conjectured from computational evidence that the geometric length and the word length of a free homotopy class of loops are correlated in the following fundamental way.
\begin{conjecture}[Chas--Li--Maskit]
For an arbitrary hyperbolic metric $\rho$ on the pair of pants that makes its boundary geodesic, let 
$\cD_k$ denote the set of free homotopy classes with word length $k$ in a fixed choice of free generators for the fundamental group. Then, there exist constants $\mu$ and $\sigma$ (depending on $\rho$) such that for any $a<b$, the proportion of   $\cC \in \cD_k$ such that the geometric length $h(\cC)$ of the unique geodesic representative of $\cC$ satisfies
\[
\frac{h(\cC)-\mu k}{\sqrt{k}} \in [a,b]
\]
converges to 
\[
\frac{1}{\sigma \sqrt{2\pi}}\int_a^b e^{ -\frac{x^2}{2\sigma^2}}dx
\]
as $k\to\infty$. In other words, as $k\to \infty$, the distribution of geometric lengths $g(\cC)$ for $\cC\in \cD_k$ approaches the Gaussian distribution with  mean $\mu k$ and standard deviation $\sigma k$.
\end{conjecture}
\noindent It is natural to expect that even when  restricting to  commutators, the resulting distribution would still exhibit a fundamental correlation between geometric length and word length. If this were proven true, then one could use Corollary~\ref{cor1} to indirectly count closed geodesics in $X$ with geometric length $\le L$ satisfying Property ($\star$).

For a general hyperbolic orbifold $\Gamma \backslash \HH$, we have the aforementioned natural correspondence between the closed geodesics of $\Gamma \backslash \HH$ and the hyperbolic conjugacy classes of $\Gamma$. In the case that $\Gamma$ is finitely generated, we claim that the number of non-hyperbolic conjugacy classes of commutators with word length $k$ is small. Indeed, $\Gamma$ only has finitely many elliptic conjugacy classes. Moreover, $\Gamma$ only has finitely many primitive parabolic conjugacy classes, corresponding to equivalence classes of cusps. Denoting the primitive parabolic conjugacy classes by $\{\cC_i\}_{1\le i \le n}$, we can conjugate  each $\cC_i$ by elements in $\Gamma$ to assume without loss of generality that the representative $g_i$ is fully cyclically reduced. Then, since every parabolic conjugacy class is a power of a primitive one, we see that the number of parabolic classes with word length $k$ is at most $n$. 

Suppose further that $\Gamma$ is isomorphic to $F_r$ (respectively, to $G_1 * G_2$ for two nontrivial finite groups $G_1$ and $G_2$) and fix a set of generators. We see that the number of non-hyperbolic conjugacy classes of commutators with length $k$ is bounded by a constant, and thus clearly dominated by the error term of $\ll k(2r-1)^{k/2}$ (respectively, $\ll k(|G_1|-1)^{k/4}(|G_2|-1)^{k/4}$).  In fact, it is known~\cite{notpower} that a commutator of a free group cannot be a proper power, which shows that the number of non-hyperbolic conjugacy classes of commutators is finite in this case. For the latter case of $\Gamma\cong G_1 * G_2$, the work of Comerford, Edmunds, and Rosenberger~\cite{maybepower} gives the possible forms that a proper-power commutator can take in a general free product. After discarding the $O(1)$ number of non-hyperbolic conjugacy classes of commutators with word length $k$, one obtains from Corollary~\ref{cor1} the asymptotic number of closed geodesics $\gamma$ with word length $k$ satisfying Property ($\star$), with the error term unchanged. It is natural to expect for $\Gamma \backslash \HH$ a correlation between the word length and the geometric length similar to that conjectured for the pair of pants. If such a correlation were proven true, then Corollary~\ref{cor1} would yield information on the number of  closed geodesics in $\Gamma \backslash \HH$ with length $\le L$ satisfying Property ($\star$).

For an example of such a hyperbolic orbifold $\Gamma \backslash \HH$ whose fundamental group $\Gamma$ is isomorphic to $G_1 *G_2$, we introduce the Hecke group $H(\ld)$ for a real number $\ld>0$, defined by the subgroup of $\PSL_2(\RR)$ generated by 
\begin{equation*}
S \defeq \begin{pmatrix}\label{matrixgenerators}
0 & -1 \\
1 & 0
\end{pmatrix} \hspace{24pt} \text{ and } \hspace{24pt} T_{\ld}  \defeq \begin{pmatrix} 
1 & \ld \\
0 & 1 
\end{pmatrix}.
\end{equation*}
Denote $\ld_q \defeq 2\cos(\pi/q)$ for integers $q \ge 3$. Hecke~\cite{hecke} proved that $H(\ld)$ is Fuchsian if and only if $\ld =\ld_q$ for some integer $q \ge 3$ or $\ld \ge 2$. Specializing to the former case, $H(\ld_q)$ has the presentation
\[
\langle S, R_\ld : S^2 = R_\ld^q = I \rangle \cong \ZZ/2\ZZ * \ZZ/q\ZZ,
\]
where the generator $R_\ld$ can be taken to be  $S T_\ld$~\cite{nexthecke}. Thus, our desired example $H(\ld_q) \backslash \HH$ has fundamental group isomorphic to $\ZZ/2\ZZ * \ZZ/q\ZZ$. In particular, note that $H(\ld_3) = \PSL_2(\ZZ)$ (with the standard generators $S$ and $T_1$), which has a presentation as the  free product $\ZZ/2\ZZ * \ZZ/3\ZZ$. 

\begin{note}
The Chas--Li--Maskit conjecture has recently been proven in a preprint~\cite{gtt} by  Gekhtman,  Taylor, and Tiozzo.
\end{note}

\subsection{Connections to number theory}
Counting closed geodesics on hyperbolic orbifolds $\Gamma \backslash \HH$ by geometric length  is a well-studied problem that often has analogies (in both the asymptotic of the resulting counting function and the techniques used to prove it) to counting problems in number theory. To illustrate, the primitive hyperbolic conjugacy classes of $\Gamma$ correspond precisely to the primitive closed geodesics of $\Gamma \backslash \HH$. These are called prime geodesics because when ordered by trace, they satisfy equidistribution theorems analogous to those of prime numbers, such as the prime number theorem (generally credited to Selberg \cite{selberg}, while its analogue for surfaces of varying negative curvature was proven by Margulis \cite{margulis}) and Chebotarev's density theorem (proven by Sarnak \cite{sarnak}).  Specifically, this analogue of the prime number theorem is called the prime geodesic theorem, which states that the number of prime geodesics $\gamma \in \Gamma \backslash \HH$ with norm $\le N$ is asymptotically given by $\sim N/\log N$, where the norm $n(\gamma)$ denotes the real number $\rho > 1$ such that the Jordan normal form of the $\PSL_2(\RR)$-matrix corresponding to $\ell$ is $\begin{psmallmatrix} 	\pm \rho	&	0	\\	0	&	\pm\rho^{-1}	\end{psmallmatrix}$, which is related to the geometric length $\ell(\gamma)$ by the relation (see~\cite[p. 384]{mac}) given by
\begin{align}\label{normtrace}
\ell(\gamma)=\log n(\gamma).
\end{align}
\iffalse
\left( \cosh \frac{\ell(\gamma)}{2} + \sqrt{\left( \cosh\frac{\ell(\gamma)}{2}\right)^2-1} \right)^2.
\fi
Furthermore, Sarnak's analogue of Chebotarev's density theorem implies that for every surjective homomorphism $\varphi: \Gamma \to G$ onto a finite abelian group  $G$, the number of prime geodesics of $\Gamma \backslash \HH$ with norm $\le N$ in the inverse image of a given $g \in G$ is asymptotically given by $\sim N/(|G| \log N)$. In particular, if $[\Gamma, \Gamma]$ is of finite index in $\Gamma$, then letting $\phi$ be the quotient map $\Gamma \to \Gamma/[\Gamma.\Gamma]$, one sees that the number of prime geodesics of $\Gamma \backslash \HH$ with norm $\le N$ and trivial abelianization is asymptotically given by 
\[
\sim \frac{N}{ \left|\Gamma : [\Gamma ,\Gamma]\right| \log N}.
\]
The result of Corollary~\ref{cor1} can be seen as the answer to an analogous problem of counting closed geodesics of $\Gamma \backslash \HH$ that correspond to conjugacy classes of commutators rather than arbitrary conjugacy classes with trivial abelianization, the main difference being that our formula counts closed geodesics ordered by word length.

We mention one more number-theoretic setting which involves counting conjugacy classes of commutators of a Fuchsian group, specialized to the modular group $\Gamma = \PSL_2(\ZZ) = H(\ld_3)$. A commutator of $\Gamma$ is precisely a coset $\{C, -C\}$ for a commutator $C=ABA^{-1}B^{-1}$ of $\SL_2(\ZZ)$, where $A, B \in \SL_2(\ZZ)$; this choice is unique because $-I$ has nontrivial abelianization, which implies that $-C$ does, as well. Thus, a conjugacy class of commutators in $\Gamma$  with such a representative $\{C, -C\}$  is the union of the $\SL_2(\ZZ)$-conjugacy class of the $\SL_2(\ZZ)$-commutator $C$ and the $\SL_2(\ZZ)$-conjugacy class of $-C$.  

For the free group $F_2$ on generators $X$ and $Y$, the map $\Hom(F_2, \SL_2(\CC)) \to \CC^3$ defined by $\rho \mapsto (\Tr \rho(X), \Tr \rho(Y),\Tr \rho(XY))$ induces an isomorphism between the moduli space of $\SL_2$-valued representations of $F_2$ and the affine space $\CC^3$, by the work of Vogt~\cite{vogt} and Fricke~\cite{fricke}. Labeling this parametrization of  $\CC^3$ as $(x,y,z)$, we note that  $\Tr \rho(XYX^{-1}Y^{-1})$ is given by the polynomial $M(x,y,z)\defeq x^2+y^2+z^2-xyz-2$; in other words, we have the trace identity~\cite[p. 337]{trace}
\begin{equation}
\label{fricke}
\Tr(A)^2+\Tr (B)^2+ \Tr (AB)^2 - \Tr( A) \Tr (B)\Tr (AB)-2 =\Tr(ABA^{-1}B^{-1})
\end{equation}
for all $A,B \in \SL_2(\CC)$. In particular, note that every conjugacy class of commutators in $\SL_2(\ZZ)$ (say, with representative $ABA^{-1}B^{-1}$ of $\SL_2(\ZZ)$ such that $A, B \in \SL_2(\ZZ)$) with a given trace $T$ gives rise to an integral solution
\[
(\Tr(A), \Tr(B), \Tr(AB))
\]
of the Diophantine equation $M(x,y,z)=T$. The integral solutions to the Markoff equation $M(x,y,z)=-2$ play an important role in the theory of Diophantine approximation through their relation to extremal numbers via the Lagrange--Markoff spectrum~\cite{markoff1, markoff2}, and have been studied recently in~\cite{sarnak1} by applying strong approximation.  Moreover, integral solutions to the Markoff-type equation $M(x,y,z)=T$ for varying $T$ are also of independent number-theoretic interest and have been studied recently in~\cite{sarnak2}. 

The conjugacy classes of $\PSL_2(\ZZ)$ with trace $T$ correspond to the $\SL_2(\ZZ)$-orbits of binary quadratic forms with discriminant $T^2-4$.  This allows one to use Gauss' reduction theory of indefinite binary quadratic forms, which yields a full set of representatives for the $\SL_2(\ZZ)$-orbits of binary quadratic forms of any positive discriminant, to obtain the conjugacy classes of $\PSL_2(\ZZ)$ with trace $T$. Furthermore, there is a straightforward condition determining whether each such conjugacy class is in the commutator subgroup $[\PSL_2(\ZZ),\PSL_2(\ZZ)]$, since $[\PSL_2(\ZZ),\PSL_2(\ZZ)]$ is precisely
\begin{align*}
\Bigg\{  \overline{\begin{pmatrix} 
a & b \\
c &  d
\end{pmatrix}} \colon   (1-c^2)&(bd+3(c-1)d+c+3)+c(a+d-3) \equiv 0 \pmod{12}  \label{congruence}
\\ \text{ or }  (1-&c^2)(bd+3(c+1)d-c+3)+c(a+d+3) \equiv 0 \pmod{12} \Bigg\} 
\end{align*}
a congruence subgroup of index $6$; this follows from the fact that $[\SL_2(\ZZ),\SL_2(\ZZ)] = \ker \chi$ for the surjective homomorphism $\chi : \SL_2(\ZZ) \to \ZZ/12\ZZ$ defined by 
\[
\begin{pmatrix} 
a & b \\
c &  d
\end{pmatrix}  \mapsto  (1-c^2)(bd+3(c-1)d+c+3)+c(a+d-3),
\]
as shown in~\cite[Proof of Theorem 3.8]{kconrad}. However, there is no known number-theoretic condition determining whether a conjugacy class of $\PSL_2(\ZZ)$ with trivial abelianization is in fact a commutator; the only known method to do so is to decompose an element of the conjugacy class into a word in the generators $S$ and $T_1$, then combinatorially checking whether there exists a cyclic conjugate of this word in the form given by Definition~\ref{wicksdef2}. For an implementation of this algorithm in \texttt{magma}, see~\cite[Appendix]{seniorthesis}. 

Since every commutator of $\SL_2(\ZZ)$ with trace $T$ gives rise to an integral point on $M_{T}$, the natural follow-up question is whether the converse is true. A local version of this holds, in that for every power $p^n$ of a prime $p \ge 5$, there exists a commutator $ABA^{-1}B^{-1}$ satisfying (\ref{fricke}) $\mod p^n$~\cite[Lemma 6.2]{sarnak2}. The global version of the converse, however, does not hold in general, nor do we know of any number-theoretic condition determining whether an integral point on $M_{T}$  comes from a commutator. Then, a question of Sarnak~\cite{sarnak3} asks:
\begin{center}
\textit{What is the asymptotic number of integral points on $M_{T}$ that arise from $\SL_2(\ZZ)$-commutators by the trace identity (\ref{fricke}) as $|T| \to \infty$?  }
\end{center}
On this front, one can consider $\SL_2(\ZZ)$-commutators as elements of $\PSL_2(\ZZ)\cong \ZZ/2\ZZ * \ZZ/3\ZZ$, for which Theorem~\ref{thm2} gives the asymptotic number of conjugacy classes of commutators with a given word length in the generators. Any correlation between the word length of a hyperbolic conjugacy class $\gamma$ of  $\PSL_2(\ZZ)$ and the geometric length of the corresponding geodesic in $\Gamma \backslash \HH$ would also yield correlation between the word length of $\gamma$ and its trace $t(\gamma)$, defined to be the absolute value of the trace of any representative matrix. This is due to the relation between $t(\gamma)$ and $\ell(\gamma)$ given by
\[
t(\gamma) = 2\cosh \frac{\ell(\gamma)}{2},
\]
which is equivalent to the relation~(\ref{normtrace}). Thus, a proof of such a correlation would allow Theorem~\ref{thm2} to yield information on the number of integral points  on  $M_{T}$ and $M_{-T}$ for $T > 2$  (the condition for hyperbolicity, where we have without loss of generality taken the representative trace  $T$ to be positive)  arising from commutators by (\ref{fricke}).

\subsection{Strategy of proof}
The main idea behind the proof of Theorem~\ref{thm1} given in Section~\ref{sec:thm1} (respectively, the proof of Theorem~\ref{thm2} given in Section~\ref{sec:thm2}) is to count the Wicks commutators of $F_r$ (respectively, those of $G_1 * G_2$) and show that each such Wicks commutator is on average cyclically conjugate to precisely six Wicks commutators (including itself) while keeping the overall error term under control. Thus, we obtain the asymptotic number of conjugacy classes of commutators with length $k$ by dividing the asymptotic number of Wicks commutators of length $k$ by $6$. Finally, in Section~\ref{sec:conclusion}, we conclude the paper with a discussion of future  directions for generalization.

\section{Proof of Theorem~\ref{thm1}}\label{sec:thm1}
\subsection{Counting the Wicks commutators of $F_r$}
Since the cyclically reduced conjugacy representative of $F_r$ is unique up to cyclic permutation, it suffices to count equivalence classes (with respect to cyclic permutation) of Wicks commutators of length $k=2X$. Let $R_X$ denote the set of reduced words of length $X$, of which there are $2r  (2r-1)^{X-1}$. For each such word $W$, the number of ways to decompose $W$ into $A, B,$ and $C$ (i.e., $W=ABC$ without cancellation) is given by the number of ordered partitions $p$ of $X$ into three (not necessarily nontrivial) parts. 

Define  $(W,p)$ to be a $\emph{viable pair of length $n$}$ if $W$ is a length-$n$ reduced word,  $p=(n_1,n_2,n_3)$ is a partition of $n$ into three parts  $n_1, n_2, n_3>0$, and the  word 
\[
W'\defeq ABCA^{-1}B^{-1}C^{-1}
\]
(obtained by decomposing $W$ as specified by $p$, in the method described above) is a Wicks commutator. \iffalse We now show that for a fixed $p$, the proportion of $W \in R_X$ such that $(W,p)$ is viable is given by 
\[
\frac{1}{2r} + O\left(\frac{1}{(2r-1)^{n_1-1}} +\frac{1}{(2r-1)^{n_2-1}} +\frac{1}{(2r-1)^{n_3-1}} \right).
\]
\fi
It is natural to first count the viable pairs $(W,p)$ of length $X$. To do so, we will need the following.

\begin{lemma}
In the set of all reduced words of length $n$ beginning with a given letter, the proportion of words ending in any one letter is given by
\[
\frac{1}{(2r-1)^{n-1}} \cdot \frac{(2r-1)^{n-1}+O(r)}{2r  } =  \frac{1}{2r} + O\left(\frac{1}{(2r-1)^{n-1}}\right).
\]
\end{lemma}
\begin{proof}
The number of cyclically reduced words with length $n$ is $(2r-1)^n+(r-1)(-1)^n+r$ (see~\cite[Proposition 17.2]{conj}). Another way of stating this is that the number $a_n$ of reduced words of length $n$ whose initial  and final letters are inverses (i.e., are not cyclically reduced) is 
\begin{align*}
a_n = 2r(2r-1)^{n-1}-(2r-1)^n&-(r-1)(-1)^n-r
\\&=(2r-1)^{n-1}-(r-1)(-1)^n-r.
\end{align*}
Similarly, let  $b_n$ denote the number of reduced words whose initial  and final letters are equal. Note that $\{a_n\}_{n\in \ZZ_{>0}}$ and $\{b_n\}_{n\in \ZZ_{>0}}$ are solutions to the recurrence 
\begin{align*}
c_n=(2r-1)c_{n-2}+(2r-2)&(2r(2r-1)^{n-3}-c_{n-2})
\\&=c_{n-2}+2r(2r-2)(2r-1)^{n-3}.
\end{align*}
Indeed, for $n\ge 3$, a reduced word whose initial and final letters are inverses is precisely of the form $sWs^{-1}$, and there are $2r-1$ possible choices for the letter $s$ if $W$ also has the aforementioned property (initial and final letters are inverses) and $2r-2$ possible choices if $W$ does not; an analogous discussion holds for the property that the initial and final letters are equal. It follows that the difference $\{b_n-a_n\}_{n\in \ZZ_{>0}}$ is a solution to the associated homogeneous recurrence $d_n=d_{n-2}$, and thus a linear combination of the functions $1$ and $n$. In light of the initial data $a_1=a_2=0$ and $b_1=b_2=2r$, we see that $b_n-a_n=2r$, which yields
\[
b_n=(2r-1)^{n-1}-(r-1)(-1)^n+r.
\]

For $s_1,s_2 \in \cS$, let $R_n^{s_1,s_2} \subset R_n$ denote the subset of words that begin with $s_1$ and end with $s_2$, which gives us a decomposition of $R_n$ into the disjoint union $R_n = \bigcup_{s_1,s_2\in \cS}R_n^{s_1,s_2}$. For any $s \in \cS$, 
\[
|R_n^{s,s^{-1}}|=\frac{a_n}{2r}=\frac{(2r-1)^{n-1}-(r-1)(-1)^n-r}{2r}
\]
and 
\[
|R_n^{s,s}|=\frac{b_n}{2r}=\frac{(2r-1)^{n-1}-(r-1)(-1)^n+r}{2r}.
\]
It follows that for $t \in \cS \setminus \{s,s^{-1}\}$,
\begin{align*}
|R_n^{s,t}| = \frac{(2r-1)^{n-1}-|R_n^{s,s^{-1}}|-|R_n^{s,s}|}{2r-2}=  \frac{1}{2r-2}\Bigg( (2r-1)^{n-1} -& \frac{(2r-1)^{n-1}-(r-1)(-1)^n}{r} \Bigg)\\&=\frac{(2r-1)^{n-1}+O(1) }{2r},
\end{align*}
which completes the proof of our claim.
\end{proof}

Now, fix a partition $p=(n_1,n_2,n_3)$ of $n$ into three parts  $n_1, n_2, n_3>0$, and consider $R_X$ with the uniform probability measure placed on its elements. Within the decomposition $W = ABC$ in accordance with $p$, let the first letter and last letter of $A$ respectively be $a_0$ and $a_1$, and define $b_0$, $b_1$, $c_0$, and $c_1$ similarly. We will compute the probability that $(W,p)$ is viable for a random $W \in R_X$, i.e., the probability that $ b_1 \neq a_0 ^{-1}$, $c_0 \neq a_0$, $c_1 \neq a_1 $, and $ c_1 \neq b_0^{-1}$. 

Suppose that $a_0=s$. Then, by our work above, the set of possible candidates for $a_1b_0$ is the $2r(2r-1)$-element set $S=\{wz : w, z \in \cS , w \neq z^{-1}\}$, each of which has probability 
\[
\frac{1}{2r-1}  \left(\frac{1}{2r}+O\left(\frac{1}{(2r-1)^{n_1-1}}\right)\right) = \frac{1}{2r(2r-1)}+O\left(\frac{1}{(2r-1)^{n_1}}\right).
\]
We fix a choice of $a_1b_0$ in $S$, and all probabilities from now on are conditional on this event. The set of possible candidates for $b_1c_0$ is also $S$, each of which has probability $1/2r(2r-1)+O(1/(2r-1)^{n_2})$. Let $S' \subset S$ be the subset of  possible candidates for $b_1c_0$ that satisfy the  conditions $a_0 \neq b_1^{-1}$ and $a_0 \neq c_0$. The cardinality of $S'$ can be computed as follows: there are $2r-1$ choices for $b_1$ satisfying $s \neq b_1^{-1}$, and conditional on this, there are $2r-2$ choices for $c_0$ satisfying $s \neq c_0$ and $b_0 \neq c_0^{-1}$, for a total of $(2r-1)(2r-2)$ elements of $S'$. Fix a choice of $b_1c_0$ in $S'$, and all probabilities from now on are conditional on this event. Since $a_1 \neq  b_0^{-1}$, the conditions  $c_1 \neq a_1 $ and $ c_1 \neq b_0^{-1}$ leave precisely $2r-2$ (out of $2r$) possible values for $c_1$, so the probability that $c_1$ satisfies these conditions is 
\[
(2r-2)\left (\frac{1}{2r} + O\left(\frac{1}{(2r-1)^{n_3-1}}\right) \right).
\]
 Overall, we have that the probability that $(W,p)$ is viable is
\begin{align*}
2r(2r-1)  \Bigg (\frac{1}{2r(2r-1)} &+ O\left(\frac{1}{(2r-1)^{n_1}}\right) \Bigg)
\\&\cdot (2r-1)(2r-2) \left (\frac{1}{2r(2r-1)} + O\left(\frac{1}{(2r-1)^{n_2}}\right) \right) \\ &\cdot (2r-2)\left (\frac{1}{2r} + O\left(\frac{1}{(2r-1)^{n_3-1}}\right) \right)\\ =
\left(\frac{2r-2}{2r}\right)^2   &\left (1+ O\left(\frac{1}{(2r-1)^{n_1-2}}\right) \right)  \left (1+ O\left(\frac{1}{(2r-1)^{n_2-2}}\right) \right) \\& \cdot \left (1+ O\left(\frac{1}{(2r-1)^{n_3-2}}\right) \right)
 \\ =\left(\frac{2r-2}{2r}\right)^2& \left (1+ O\left(\frac{1}{(2r-1)^{n_1-2} }+\frac{1}{(2r-1)^{n_2-2} }+\frac{1}{(2r-1)^{n_3-2} }\right) \right).
\end{align*}
Thus, the number of viable pairs $(W,p)$ of length $X$ is  given by
\begin{align*}
\sum_{n_1+n_2+n_3 = X}  &2r (2r-1)^{X-1}\left(\frac{2r-2}{2r}\right)^2 \\& \hspace{50pt}\cdot\left (1+ O\left(\frac{1}{(2r-1)^{n_1-2} }+\frac{1}{(2r-1)^{n_2-2} }+\frac{1}{(2r-1)^{n_3-2} }\right) \right)
\\ &=\frac{(2r-2)^2(2r-1)^{X-1}}{2r}   \Bigg( \frac{(X-2)(X-1)}{2} 
\\ &\hspace{50pt}+\sum_{n_1+n_2+n_3 = X} O\left(\frac{1}{(2r-1)^{n_1-2}} +\frac{1}{(2r-1)^{n_2-2}} +\frac{1}{(2r-1)^{n_3-2}} \right)  \Bigg) \\
 &=\frac{(2r-2)^2(2r-1)^{X-1}}{2r} 
  \left(\frac{(X-2)(X-1)}{2} + 3 \cdot O\left(\sum_{n_1+n_2+n_3 = X}  \frac{1}{(2r-1)^{n_1-2}} \right) \right)
  \\ &= \frac{(2r-2)^2(2r-1)^{X-1}}{2r}
 \Bigg(\frac{(X-2)(X-1)}{2} 
 \\&\hspace{50pt}+ O\left(\sum_{n_1=1}^{X-2}  (X-1-n_1)\frac{1}{(2r-1)^{n_1-2}} \right) \Bigg)
   \\ &= \frac{(2r-2)^2(2r-1)^{X-1}}{2r}\Bigg(\frac{(X-2)(X-1)}{2} 
   \\&\hspace{50pt}+ O\left(\frac{(2r-1)^2\left((2r-1)^{2-X}+X(2r-2)-4r+3\right)}{(2r-2)^2} \right) \Bigg)
  \\ &=\frac{(2r-2)^2(2r-1)^{X-1}}{4r} \left(X^2 + O\left( \frac{r^2}{(2r-1)^{X}}+rX\right)  \right),
\end{align*}
where all sums with $n_1+n_2+n_3 = X$ are taken over integers $n_1,n_2,n_3>0$.

There could \textit{a priori} be a Wicks commutator $ABCA^{-1}B^{-1}C^{-1}$ arising from distinct viable pairs, say $(W, p_1)$ and $(W, p_2)$ for $p_1=(n_1,n_2,n_3)$, and $p_2 = (m_1,m_2,m_3)$. We show that the number of such commutators is small. 

\iffalse
Label the letters of $A$ as $a_1, \ldots, a_{n_1}$, and label the letters of $B, C, D, E,$ and $F$ similarly. We have that $w_{1}, w_{2},$ and $w_{3}$ as subwords comprised of the letters 
\begin{equation*}\label{letters}
d_{\ell}^{-1}, \ldots, d_1^{-1}, a_1, \ldots, a_{n_1}, b_1, \ldots, b_{n_2},  d_1, \ldots, d_\ell, e_1, \ldots, e_{n_3-2\ell}.
\end{equation*}
Then, note that the second half of $W'$ can be considered in two forms:
\[
F  A^{-1} B^{-1}  F^{-1}E^{-1}= w_{1}^{-1} w_{2}^{-1} w_{3}^{-1} . 
\]
Equivalently, this equality can be written as 
\begin{equation*}\label{equality}
EF B  A F^{-1} = w_3 w_2  w_1,
\end{equation*}
where we reiterate that we substitute the letters of (\ref{letters}), in the correct order, for $w_3, w_2,$ and $w_1$.

Consider the function $f$ mapping the ordered set of symbols of the left-hand side of (\ref{equality}),
\[
\cA\defeq \{e_1,\ldots,e_{n_3-2\ell}, f_1,\ldots,f_{\ell},b_1,\ldots,b_{n_2}, a_1,\ldots,a_{n_1}, f_{\ell}^{-1},\ldots,f_1^{-1}\},
\]
to the ordered set $\cB$ of symbols of the right-hand side of (\ref{equality}), defined so that it maps the $i$th leftmost letter of the left-hand side to the $i$th leftmost letter of the right-hand side. 
\fi

Let $W=ABC$ be its decomposition with respect to $p_1$, and $W=A'B'C'$ its decomposition with respect to $p_2$.  Consider the function $f : \{1, \ldots, X\} \to  \{1, \ldots, X\}^2$ that maps $i$ to $(j,k)$ in the following way: the $i$th letter of $A^{-1}B^{-1}C^{-1}=A^{\prime -1}B^{\prime -1}C^{\prime -1}$, when viewed in terms of the decomposition given by $p_1$, is the inverse of the $j$th letter of $W$, and  when viewed in terms of the decomposition given by $p_2$, is the inverse of the $k$th letter of $W$. For example, the first letter of $A^{-1}B^{-1}C^{-1}$ is defined to be the inverse of the $n_1$th letter when the decomposition is in terms of $p_1$, and is defined to be the $m_1$th letter of $W$ when in terms of $p_2$, so $f(1) = (n_1,m_1)$. We consider two cases: when the two entries of $f(i)$ are distinct for all $i$, and otherwise. In the first case, the following algorithm allows us to reduce the degrees of freedom for the letters of $W$ by at least half.
\begin{enumerate}
\item Let $i=1$. For $f(i)=(j_i,k_i)$, do the following:
\begin{itemize}
\item If neither the $j_i$th or the $k_i$th position has an indeterminate variable assigned to it, then assign a new indeterminate variable simultaneously to the $j_i$th and $k_i$th positions. This increases the number of indeterminate variables by two.

\item If just one of the $j_i$th and the $k_i$th positions has an indeterminate variable assigned to it, but the other does not, then assign this indeterminate variable to the former. 

\item If both the $j_i$th and the $k_i$th positions have the same indeterminate variable assigned to them, then make no changes. 

\item If the $j_i$th and the $k_i$th positions have distinct indeterminate variables assigned to them, then set these indeterminate variables equal to each other. This decreases the number of indeterminate variables by one.
\end{itemize}
\item Incrementing $i$ by one each time,  repeat this procedure for all $1 \le i \le X$. 
\end{enumerate}
By our hypothesis that the two entries of $f(i)$ are distinct for all $i$, the number of indeterminate variables, which precisely represents the number of degrees of freedom for the word $W$ such that $(W,p_1)$ and $(W,p_2)$ give rise to the same commutator, is $\le X/2$. It follows that there are only $O((2r-1)^{X/2})$ of such words for each pair $p_1, p_2$. 

Now, consider the next case that there exists an $i$ such that the two entries of $f(i)$ are equal. Then, we consider the following three cases for $W$, $p_1$, and $p_2$.

\emph{Case 1}. Suppose the smallest $i$ such that the two entries of $f(i)$ are equal satisfies that this entry is a position in $A$. Then, $n_1=m_1$ must be this entry and $i$ must equal $1$, since otherwise we can continue to decrement $i$ so that the two entries of $f(i)$ are incremented and remain equal, a contradiction. Next, the subwords $B^{-1}C^{-1}$ and $B'^{-1}C'^{-1}$ must be equal. Without loss of generality, suppose that $n_2 > m_2$. Decompose $B=B'D$ so that our condition $B^{-1}C^{-1}=B'^{-1}C'^{-1}$ is precisely $CB'D = DCB'$. Since $CB'$ and $D$ commute, we have that they are both powers of a common subword $V$. \iffalse Since $n_2>m_2$, we have that $D$ is a positive power, and likewise $CB'$ is a positive power 
both of them must be positive powers, since otherwise there would be cancellation, which contradicts that the commutator is cyclically reduced. \fi We can bound the number of $(W,p)$ satisfying this case by counting, for each $i=X-n_1$  and for each  divisor $d \mid i$ (denoting the length of $V$), the number of ways to divide the right subword of length $i$ into two subwords $B$ and $C$, and the number of degrees of freedom. Thus, the number of double-countable Wicks commutators in  this case can be bounded from above by
\begin{align*}
\sum_{i=1}^{X-1}\sum_{d \mid i} (i-d+1)& (2r-1)^{X-i+d} \\&=
(X-1)(2r-1)^X +\sum_{i=1}^{X-1}\sum_{d \mid i,  d \neq i } (i-d+1) (2r-1)^{X-i+d} 
\\ &\le  X(2r-1)^X +(2r-1)^X \sum_{i=2}^{X-1} \sum_{1\le d \le \frac{i}{2}} (i-d+1)  (2r-1)^{-i+d} 
\\&=  X(2r-1)^X +(2r-1)^X  \sum_{i=2}^{X-1} (2r-1)^{-i} \sum_{1\le d \le \frac{i}{2}} (i-d+1)  (2r-1)^{d} 
\\ &\le X(2r-1)^X+ (2r-1)^X \sum_{i=2}^{X-1} (2r-1)^{-i+1} 
\\&\hspace{30pt} \cdot \frac{ i (2r-2)\left( (2r-1)^{\frac{i}{2}} -2\right) +2(2r-1) \left( (2r-1)^{\frac{i}{2}} -1 \right)}{2(2r-2)^2}\\ &\ll X(2r-1)^{X} ,
\end{align*}
\noindent which is dominated by our error term. 

\iffalse \\&=  (2r-1)^X \cdot \frac{1}{2(2r-2)^3(2r-1)^{X}} \Big( (X-2)(2r-1)^{\frac{X}{2} +2} - (X+2)(2r-1)^{\frac{X}{2} +3} 
\\ &\hspace{20pt} -(2r-1)^{X+\frac{1}{2}} - 6(2r-1)^{X+1} + 5(2r-1)^{X+\frac{3}{2}} +4(2r-1)^{X+2} + (X-1)(2r-1)^{\frac{X+3}{2}} 
\\&\hspace{20pt}-(X+3)(2r-1)^{\frac{X+5}{2}} +2(2r-1)^X + 2(X+1)(2r-1)^3 -2(X-1)(2r-1)^2\Big), \fi

\emph{Case 2}. Suppose the smallest $i$ such that the two entries of $f(i)$ are equal satisfies that this entry is a position in $C$. An argument symmetric to the one above can be given to show that the above expression is also an upper bound for the number of  $(W,p)$ satisfying this case.

\emph{Case 3}. Suppose the smallest $i$ such that the two entries of $f(i)$ are equal satisfies that this entry is a position in $B$. Without loss of generality, suppose $n_1>m_1$. Then, $f(m_1+1)=(n_1-m_1,m_1+m_2)$, and by an argument similar to that in Case $(1)$, we have that the simultaneous entry of the aforementioned $f(i)$ must be $n_1+n_2$, with $n_1-m_1=n_3-m_3$ so that $f(m_1+1)=(n_1-m_1,n_1+n_2+(n_1-m_1))$. Thus, divide $W$ into $DEFGH$ so that $|D|+|E| = n_1$, $|G|+|H|=n_3$, and $|E|=|G|$. Then, $(W,p_1)$ gives rise to the commutator $WE^{-1}D^{-1}F^{-1}H^{-1}G^{-1}$, while $(W,p_2)$ gives rise to the commutator $WD^{-1}G^{-1}F^{-1}E^{-1}H^{-1}$. Since these are equal, it follows that $DE =GD$ and $GH=HE$. Note that if a word $\square$ satisfies the equality $\square  E= G \square$ without cancellation, then $\square$ is uniquely determined, since one can inductively identify the letters of $\square$ from left to right (or right to left). It follows that $D=H$. But this contradicts the assumption that $DEFGHD^{-1}G^{-1}F^{-1}E^{-1}H^{-1}$ is cyclically reduced.

Note that pairs $(W,p)$ such that $n_i=0$ for some $i \in \{1,2,3\}$ are counted in the above cases. This justifies our assumption of $n_1,n_2,n_3 >0$ in our earlier counting of the main term by passing to the task of counting viable pairs of length $X$, which \textit{a priori} only accounts for Wicks commutators $ABCA^{-1}B^{-1}C^{-1}$ such that $|A|,|B|,|C|>0$. We have thus shown the following.
\begin{lemma}
The total number of Wicks commutators with length $2X$ is given by
\[
\frac{(2r-2)^2(2r-1)^{X-1}}{4r} \left(X^2 + O_r\left( X\right)  \right).
\]
\end{lemma}

\subsection{A conjugacy class of commutators  contains six Wicks commutators on average}
We need to count the number of  conjugacy classes containing at least one Wicks commutator. Consider the conjugacy class $\cC$ of the Wicks commutator $W'=ABCA^{-1}B^{-1}C^{-1}$ arising from $(W,p)$, where $p=(n_1,n_2,n_3)$ with $n_1,n_2,n_3>0$; recall that the Wicks commutator not satisfying this latter hypothesis are negligible. Note that the minimum-length elements in a conjugacy class are precisely the cyclically reduced words, and that two cyclically reduced words are conjugate if and only if they are cyclically conjugate. The Wicks commutators 
\begin{align*}
BCA^{-1}B^{-1}C^{-1}A, CA^{-1}B^{-1}C^{-1}AB, &A^{-1}B^{-1}C^{-1}ABC, 
\\&B^{-1}C^{-1}ABCA^{-1}, \text{ and } C^{-1}ABCA^{-1}B^{-1}
\end{align*}
are conjugates of $W'$. We show that the number of other Wicks commutators in $\cC$ is on average negligible.

For an arbitrary $1 \le \ell \le n_3/2$ denoting the number of letters of the conjugation,  let $C=DEF$ be a decomposition without cancellation such that $|D|=|F|= \ell$. Label the letters of $W'$ by $A=a_1 \cdots a_{n_1}$, $B=b_1 \cdots b_{n_2}$, $D=d_1 \cdots d_\ell$, $E=e_1\cdots e_{n_3-2\ell}$, and $F =f_1\cdots f_\ell$.  Consider the cyclic conjugate 
\[
W''\defeq D^{-1}ABDEFA^{-1}B^{-1}F^{-1}E^{-1}
\]
of $W'$. We wish to show that on average, $W''$ is not a Wicks commutator. Suppose the  contrary, that there exists a partition $p'=(m_1,m_2,m_3)$ of $X$ into three parts such that
\[
W''=D^{-1}ABDEFA^{-1}B^{-1}F^{-1}E^{-1}= w_1w_2w_3w_1^{-1}w_2^{-1}w_3^{-1}
\]
for subwords $w_1, w_2,$ and $w_3$ of lengths $m_1, m_2,$ and $m_3$.

Label the letters of $A$ from left to right as $a_1, \ldots, a_{n_1}$, and label the letters of $B,C,D,E,$ and $F$ similarly. We have that $w_{1}, w_{2},$ and $w_{3}$ are subwords comprised of the letters 
\begin{equation}\label{letters1}
d_{\ell}^{-1}, \ldots, d_1^{-1}, a_1, \ldots, a_{n_1}, b_1, \ldots, b_{n_2}, d_1, \ldots, d_\ell, e_1, \ldots, e_{n_3-2\ell}.
\end{equation}
Then, note that the second half of $W'$ can be considered in two forms:
\[
F A^{-1}  B^{-1}  F^{-1}E^{-1}= w_{1}^{-1}  w_{2}^{-1}  w_{3}^{-1} . 
\]
Equivalently, this equality can be written as 
\begin{equation}\label{equality1}
EF B A F^{-1} = w_3 w_2 w_1,
\end{equation}
where we reiterate that we substitute the letters of (\ref{letters1}), in the correct order, for $w_3, w_2,$ and $w_1$.

Consider the function $g$ mapping the ordered set of symbols of the left-hand side,
\[
\cA\defeq \{e_1,\ldots,e_{n_3-2\ell}, f_1,\ldots,f_{\ell},
b_1,\ldots,b_{n_2},a_1,\ldots,a_{n_1}, f_{\ell}^{-1},\ldots,f_1^{-1}\},
\]
to the  ordered set $\cB$ of symbols of the right-hand side of (\ref{equality1}), which are just the symbols of (\ref{letters1}) reordered appropriately. Specifically, $g$  maps the $i$th leftmost letter of the left-hand side of (\ref{equality1}) to the $i$th leftmost letter of the right-hand side. 

First, suppose $g$ has no fixed points ($\mathfrak i$ such that $g(\mathfrak i)=\ii$). Then, use an algorithm similar to the previous one to conclude that there are $\le X/2$ degrees of freedom for $ABDEF$, so $W$ must be one of only $O((2r-1)^{X/2})$ choices (for each choice of $\ell$ and $p'$).

Now, suppose that there exists an $\mathfrak i$ such that $g(\mathfrak i)=\ii$. Such fixed points $\ii$ must be letters of $A$, $B$, or $E$. We first consider the case that all the fixed points are letters of only one of $A$, $B$, and $E$. In this case, we consider the following subcases for $W$, $p$, and $p'$: 

\emph{Case 1}. Suppose that the fixed points are letters of $E$. Then, all of the fixed points must be in one of $w_2$ and $w_3$; they cannot be in $w_1$, since this would mean that $w_1$ contains $e_1$, but $e_1$ is then necessarily located at different positions in the left-hand side and right-hand side of (\ref{equality1}). Suppose that the fixed points of $E$ are in $w_3$. Then, in order for the letters of $E$ to match, we require that $w_3=E$. This means that $g(f_1)$ is the first letter of $w_2$, which is adjacent to the last letter of $w_1$. But the last letter of $w_1$ is $g(f_1^{-1})$, which shows that we have adjacent letters that are inverses. This contradicts the fact that $W'$ is  cyclically reduced. 

Next, suppose all the fixed points are in $w_2$. Then, we  must have that $m_3=n_3-2\ell-(m_2+m_3)$ so that the first letter of $w_2$ is at the same position in both the left-hand and right-hand side. Thus, $m_2+2m_3=n_3-2\ell$, which means there are $\le (n_3-2\ell)/2$ choices for $p'$ parametrized by $m_3 \le(n_3-2\ell)/2$ . For each such choice of $p'$, there are $m_2=n_3-2\ell-2m_3$ fixed letters,  from $e_{m_3+1}$ to $e_{n_3-2\ell-m_3}$, and $(X-(n_3-2\ell-2m_3))/2$ non-fixed letters. Counting across all choices of values for the letters, $p$, $p'$, and $\ell$,  we have that the number of additional Wicks commutators arising from this case is bounded from above by
\begin{align*}
&\sum_{n_1=0}^{X} \sum_{n_2=0}^{X-n_1} \sum_{\ell=1}^{\floor{\frac{X-n_1-n_2}{2}}} \sum_{m_3=0}^{\floor{\frac{n_3-2\ell}{2}}}  (2r-1)^{\frac{X+n_3-2\ell-2m_3}{2}} \ll (2r-1)^X,
\end{align*}
which is dominated by our error term.

\emph{Case 2}.  Suppose that the fixed points are letters of $B$. Then, all of the fixed points must be in one of $w_1$, $w_2$, and $w_3$. First, suppose they are in $w_1$. Then, note that $g(f_{\ell})=a_{n_1}$, but we also have $g(a_{n_1})$ is next to $g(f_{\ell}^{-1})$, which leads to the contradiction that a letter cannot equal its inverse. Second, suppose the fixed letters are in $w_3$. Then, $g(d_{1})=a_1$, but $a_1$ is adjacent to $d_1^{-1}$, a contradiction. 

Thus, the fixed points of $B$ must be in $w_2$. We consider three subcases: $n_2 > m_2$, $n_2< m_2$, and $n_2=m_2$. If $n_2 > m_2$, then  in order for the letters of $B$ to match, we require that the leftmost fixed letter of $B$ is $b_{\frac{n_2-m_2}{2}+1}$.  But then $b_{\frac{n_2-m_2}{2}}$ is both equal to $f_{1}^{-1}$ (since $g(f_1^{-1})=b_{\frac{n_2-m_2}{2}}$) and $e_{n_3-2\ell}$ (since $g(b_{\frac{n_2-m_2}{2}})=e_{n_3-2\ell}$) as letters of $\cG$, which contradicts the fact that $f_1$ and $e_{n_3-2\ell}$ are adjacent. If $n_2 < m_2$, then $g(f_{\ell})=a_{n_1}$, but also $g(a_{n_1})$ is the letter in $D^{-1}A$ that is left of the letter $g(f^{-1}_{\ell})$, giving us the contradiction that the letter in $\cG$ in the position $f_{\ell}$ is adjacent to the letter in the position $f_{\ell}^{-1}$. This implies that $n_2=m_2$, from which we can use an argument similar to that in Case 1 of the previous casework (showing that $W'$ on average can be only decomposed as a commutator in one way) to conclude that $A$ is a power of $D^{-1}$ and $E$, a power of $D$. It follows that our original $(W,p)$ is one of the pairs falling under Case 1 of the previous casework, which are negligible. 

\emph{Case 3}. Suppose that the fixed points are letters of $A$. Then, all of the fixed points must be in one of $w_1$ and $w_2$; they cannot be in $w_3$, since then there must be more than $\ell$ letters right of $A$. Suppose the fixed letters of $A$ are in $w_1$. Then, we must have $g(b_{n_2})=d_1^{-1}$, which contradicts the fact that $b_{n_2}$ is adjacent to $d_{1}$.  Therefore,  the fixed points are necessarily in $w_2$. This requires that $m_1-\ell=n_1+\ell-(m_1-m_2)$ in order for the letters of $A$ to be in matching positions. Thus, we have $m_2=n_1+2\ell-2m_1$.  Note then that $p'$ is parametrized by $m_1 \le (n_1+2\ell)/2$. For each choice of $p'$, we have $n_1$ fixed letters (and $(X+n_1)/2 \le (X+n_1-m_1+\ell)/2$ overall degrees of freedom) if $m_1 \le \ell$, and  $n_1-(m_1-\ell)$ fixed letters (and $(X+n_1-m_1+\ell)/2$ overall degrees of freedom) if $m_1 > \ell$. Thus, counting across all choices of values for the letters, $p$, $p'$, and $\ell$,  we have that the number of additional Wicks commutators arising from this case is bounded from above by
\begin{align*}
 \sum_{n_1=0}^{X} \sum_{n_2=0}^{X-n_1}  \sum_{\ell=1}^{\floor{\frac{X-n_1-n_2}{2}}} \sum_{m_1=0}^{\floor{\frac{n_1+2\ell}{2}}} (2r-1)^{\frac{X+n_1+\ell-m_1}{2}} \ll (2r-1)^X,
\end{align*}
which is dominated by our error term.

Next, consider the case where the fixed points are in two of $A, B,$ and $E$. It is necessary that the fixed letters inside these two subwords must respectively be in two distinct subwords among $w_1,w_2$, and $w_3$. However, we have shown above that the subwords $w_1$ and $w_3$ cannot contain fixed points, a contradiction. Finally, the fixed letters cannot be in all of $A, B,$ and $E$. Indeed, if this were true, then in order for the letters of $A$ and $E$ to match, we require $m_1=n_1+2\ell$ and $m_3=n_3$. But then the letters of $B$ cannot possibly match, a contradiction. 

If $\ell \ge n_3/2$, then we can think of our commutator as a cyclic conjugate of $C^{-1}ABCA^{-1}B^{-1}$ such that the letters are moved from left to right. A symmetric argument like above gives us the same conclusion for this case. We have thus shown that the number of conjugacy classes of commutators with length $2X$ is given by
\[
\frac{1}{6} \cdot \frac{(2r-2)^2(2r-1)^{X-1}}{4r} \left(X^2 + O_r\left(X\right)  \right)= \frac{(2r-2)^2(2r-1)^{X-1}}{24r} \left(X^2 + O_r\left( X\right)  \right),
\]
as needed.

\section{Proof of Theorem~\ref{thm2}}\label{sec:thm2}
\subsection{Counting the Wicks commutators of $G_1 * G_2$}
In addition to his theorem classifying commutators of free groups, Wicks~\cite{wicks} also proved the following analogous theorem characterizing all commutators of a free product of arbitrary groups.
\begin{theorem}[Wicks]\label{wicksfp}
A word in $\Conv_{i\in I} G_i$ is a commutator if and only if it is a conjugate of one of the following fully cyclically reduced products:
\begin{enumerate}[label=(\alph*)]
\item a word comprised of a single letter that is a commutator in its factor $G_i$,
\item $X\al_1 X \al_2^{-1}$, where $X$ is nontrivial and $\al_1, \al_2$ belong to the same factor $G_i$ as conjugate elements,
\item $X \al_1 Y \al_2 X^{-1} \al_3 Y^{-1} \al_4$, where $X$ and $Y$ are both nontrivial, $\al_1, \al_2, \al_3,\al_4$ belong to the same factor $G_i$, and $\al_4\al_3\al_2\al_1$ is trivial,
\item $XYZX^{-1}Y^{-1}Z^{-1}$,
\item $X  Y \al_1  Z  X^{-1} \al_2 Y^{-1}  Z^{-1} \al_3$, where $Y$ and at least one of $X$ and $Z$ is nontrivial,  $\al_1, \al_2, \al_3$ belong to the same factor $G_i$, and $\al_3\al_2\al_1$ is trivial,
\item $X \al_1 Y \be_1  Z \al_2 X^{-1} \be_2 Y^{-1} \al_3 Z^{-1} \be_3$, where $\al_1, \al_2, \al_3$ belong to the same factor $G_i$ and $\be_1, \be_2,\be_3$, to $G_j$, $\al_3\al_2\al_1=\be_3\be_2\be_1=1$, and either $\al_1, \al_2, \al_3, \be_1, \be_2,\be_3$ are not all in the same factor or $X,Y,Z$ are all nontrivial.
\end{enumerate}
\end{theorem}
Note that in the above, the Greek letters are assumed to be nontrivial. This convention  is used later in our proof, where when a Greek letter $\al$ is said to satisfy $\al \in G_i$, we mean that $\al$ is a nontrivial element of $G_i$. 

Equipped with a complete classification of commutators in an arbitrary free product, we proceed with our proof of Theorem~\ref{thm2}. Let $G_1 * G_2$ be the free product of two nontrivial finite groups $G_1$ and $G_2$. The letters of a cyclically reduced word of $G$ must alternate between equal numbers of elements of $\gone$ and elements of $\gtwo$, and thus must be of even length. Let $C$ be a fully reduced commutator of $G_1 * G_2$. If $C$ is of the form $(d)$ listed in Theorem~\ref{wicksfp} with none of $X,Y,$ and $Z$ trivial, then the last letter of $X$ is in different factors compared to the first letter of $Y$ and the last letter of $Z$, which must also be in different factors, a contradiction. If $C$ is of the form $(e)$, then $\al_1, \al_2,$ and $\al_3$ must be in the same free-product factor, but this would imply that the last letter of $Y$ is in different free factors compared to the first letter of $Z$  and the first letter of $X$; this contradicts the similar implication that the first letter of $Z$  and the first letter of $X$ are in different factors. Thus, $C$ must be of the form $(a)$, $(b)$, $(c)$, $(f)$, or $XYX^{-1}Y^{-1}$ (i.e., of the form $(d)$ with $|Z|=0$). However, if $C$ is of the form $(f)$, then $G_i$ and $G_j$ must be the same free-product factor, since otherwise $\al_1$ would be adjacent to letters of distinct free-product factors, a contradiction.

By Wicks' theorem for free products, we need to count cyclic conjugacy classes of Wicks commutators of $G_1 * G_2$. It follows from the discussion above that every commutator of $W$ is conjugate to one of the fully cyclically reduced forms listed under Definition~\ref{wicksdef2}. These general forms, labeled from $(1)$ to $(9)$, have without loss of generality been taken to have the letters of odd position be in $G_1$ and those of even position in $G_2$. Throughout this proof, we will regularly use the terminology \emph{Wicks commutators of the form $(i)$}, where $1 \le i \le 9$, to refer to the corresponding general form listed under Definition~\ref{wicksdef2}.

Consider Wicks commutators of $G_1 * G_2$ with length $k$, where $k$ is a multiple of $4$. Let $X=k/4$, so that the left-half subword of $W$ contains $X$ letters of the $G_1$ factor and $X$ letters of the $G_2$ factor, which are placed in  alternating order.  The commutators of the forms $(1)$ and $(2)$ are $O(1)$ in number and all have length $1$. The commutators of the forms $(3)$ and $(4)$ are $O\left(\goc^{X}  \gtc^{X}\right)$ in number, since there are $\le X$ degrees of freedom among the letters of each free-product factor. The commutators of the forms $(5), (6),$ and $(7)$ are $O\left(X  \goc^{X}  \gtc^{X}\right)$ in number, since there are $\le X$ degrees of freedom for the $X$ letters of each free-product factor and $O(X)$ possible pairs of values for $|A|$ and $|B|$, by an argument similar to that used in Section~\ref{sec:thm1}.

The number of Wicks commutators of the form $(8)$ is
\[
  \frac{(X-5)(X-4)}{2}  \left( \left|G_2\right|-2 \right)^2  \goc^X \gtc^{X-1}.
\]
Indeed, we have $\gtc \left( \left|G_2\right|-2 \right)$ distinct choices of the triple $(\al_1, \al_2, \al_3)$, since there are $\left|G_2\right|-1$  choices for $\al_1$, $\left|G_2\right|-2$ choices for $\al_2 \neq \al_1^{-1}$, and $\al_3$ is uniquely determined from the previous choices.  Likewise, there are $\gtc \left( \left|G_2\right|-2 \right)$ distinct choices of the triple $(\be_1, \be_2, \be_3)$. Finally, there are $(X-5)(X-4)/2$ partitions of $X-3$ into three nontrivial parts $(n_1,n_2,n_3)$ such that  $|A| = 2n_1 +1$, $|B|= 2n_2 +1$, and $|C|= 2n_3 +1$; $X$ degrees of freedom for choosing the $G_1$-letters of $A, B$, and $C$; and $X-3$ degrees of freedom for choosing the $G_2$-letters of $A, B$, and $C$. By an analogous argument, the number of Wicks commutators of the form $(9)$ is 
\[
  \frac{(X-5)(X-4)}{2}\left( \left|G_1\right|-2 \right)^2   \goc^{X-1} \gtc^X.
\]
We define a \emph{generic Wicks commutator of $G_1 * G_2$} to be one of the form $(8)$ or $(9)$. These comprise the main term of the total number of Wicks commutators of length $k=4X$, since we have seen above that the non-generic Wicks commutators, those of the forms $(1)-(7)$, comprise a negligible subset. Overall, we have shown the following.
\begin{lemma}
The total number of Wicks commutators with length $4X$ is given by
\begin{align*}
\frac{1}{2}\left(\goc\left( \left|G_2\right|-2 \right)^2 +\left( \left|G_1\right|-2 \right)^2 \gtc \right) & X^2  \goc^{X-1} \gtc^{X-1} \\ +&O\left( X\goc^{X} \gtc^{X}  \right).
\end{align*}
\end{lemma}

\subsection{A conjugacy class of commutators contains six Wicks commutators on average} We need to count the number of conjugacy classes containing at least one generic Wicks commutator. As before, let $\cC$ be the conjugacy class of the Wicks commutator $W\defeq A\al_1 B\be_1 C\al_2 A^{-1}\be_2 B^{-1}\al_3 C^{-1}\be_3$, with $(n_1,n_2,n_3)$ a partition of $X-3$ and $|A|=2n_1+1$, $|B|=2n_2+1$, and $|C|=2n_3+1$. We wish to show that on average, $\cC$ does not contain generic Wicks commutators other than the six obvious ones:  
\begin{align*}
W, B\be_1 C\al_2 A^{-1}\be_2 B^{-1}\al_3 C^{-1} &\be_3 A \al_1, C\al_2 A^{-1}\be_2 B^{-1}\al_3 C^{-1}\be_3 A \al_1 B \be_1, 
\\
&A^{-1}\be_2 B^{-1}\al_3 C^{-1} \be_3 A \al_1 B \be_1 C\al_2, B^{-1}\al_3 C^{-1}\be_3 A\al_1B\be_1 C\al_2 A^{-1} \be_2,
\\&\text{and } C^{-1}\be_3 A\al_1 B \be_1 C \al_2A^{-1} \be_2 B^{-1}\al_3.    
\end{align*}
We wish to show that any other cyclic conjugate $W'$ of $W$ is not on average a Wicks commutator. 

Suppose the contrary. One of the ways this can happen is if there is such a $W'$ that is of the form $(8)$, i.e, there exists a partition $p'=(m_1,m_2,m_3)$ of $X-3$ into three parts such that 
\[
W'=w_{1}\al'_1 w_{2} \be'_1 w_{3} \al'_2 w_{1}^{-1} \be'_2 w_{2}^{-1} \al'_3 w_{3}^{-1} \be'_3
\]
 for subwords $w_{1}, w_{2},$ and  $w_{3}$ of lengths $2m_1+1, 2m_2+1, $ and $2m_3+1$ given by $p'$, and 
 \[
 \al'_1,\al'_2,\al'_3, \be'_1,\be'_2, \be'_3 \in G_2.
 \]
 We will see that the argument showing that these exceptions are negligible also shows that the exceptions such that $W'$ is of the other Wicks-commutator forms are also negligible.

Suppose the number of letters of the conjugation is $2\ell$, where $\ell \le (n_3+1)/2$ is arbitrary. For the desired uniformity of our presented argument, we assume that $\ell >0$, although the exceptions in the $\ell=0$ case can be bounded similarly. We decompose $C=DEF$  without cancellation so that $|D|=|F|= 2\ell-1$. Label the letters of $W$ by $A=a_1 \bar{a}_1a_2 \bar{a}_2 \cdots  \bar{a}_{n_1} a_{n_1+1}$, $B=b_1 \bar{b}_1b_2 \bar{b}_2\cdots \bar{b}_{n_2} b_{n_2+1}$, $D=d_1\bar{d}_1 \cdots \bar{d}_{\ell-1} d_{\ell}$, $E=\bar{e}_1 e_2 \bar{e}_2  \cdots   e_{n_3-2\ell +2} \bar{e}_{n_3-2\ell+2}$, and $F =f_1 \bar{f}_1 \cdots  \bar{f}_{\ell-1} f_{\ell}$; note that the barred letters denote the $G_2$-letters and the non-barred letters, the $G_1$-letters. Consider the cyclic conjugate 
\[
W'\defeq D^{-1}\be_3 A\al_1 B\be_1 D E F\al_2 A^{-1}\be_2 B^{-1} \al_3 F^{-1}E^{-1}
\]
of $W$. We wish to show that on average, $W'$ is not a Wicks commutator.

The subwords $w_{1}, w_{2},$ and $w_{3}$ are  comprised of the letters 
\begin{equation}\label{letters2}
d_{\ell}^{-1}, \ldots, d_1^{-1}, \beta_3, a_1, \ldots, a_{n_1+1}, \alpha_1, b_1, \ldots, b_{n_2+1}, \be_1, d_1, \ldots, d_\ell, \bar{e}_1, \ldots, \bar{e}_{n_3-2\ell+2},
\end{equation}
except three of these letters instead correspond to $\al'_1,\al'_2,$ and $\al'_3$ and thus omitted. Then, note that the second half of $W'$ can be considered in two forms:
\[
F \al_2 A^{-1} \be_3 B^{-1} \al_3 F^{-1}E^{-1}= w_{1}^{-1} \be'_2 w_{2}^{-1} \al'_3 w_{3}^{-1} \be'_3. 
\]
Equivalently, this equality can be written as 
\begin{equation}\label{equality}
EF\al_3^{-1} B \be_3^{-1} A \al_2^{-1} F^{-1} = \be^{\prime -1}_3w_3\al^{\prime -1}_3 w_2 \be^{\prime -1}_2 w_1,
\end{equation}
where we reiterate that we substitute the appropriate letters of (\ref{letters2}), in the correct order, for $w_3, w_2,$ and $w_1$.

Consider the function $g$ mapping the ordered set of symbols of the left-hand side of (\ref{equality}),
\[
\cA\defeq \{\bar{e}_1,\ldots,\bar{e}_{n_3-2\ell+2}, f_1,\ldots, f_{\ell},\al_3^{-1}, b_1,\ldots,b_{n_2+1}, \be_3^{-1},a_1,\ldots,a_{n_1+1},\al_2^{-1} ,f_{\ell}^{-1},\ldots,f_1^{-1}\},
\]
to the ordered  set $\cB$ of symbols of the right-hand side of (\ref{equality}), which when reordered are comprised of the letters of (\ref{letters2}) except we replace the $(2m_1+2)$th, $(2m_1+2m_2+4)$th, and $(2m_1+2m_2+2m_3+6)$th leftmost letters in (\ref{letters2}) with $\be^{\prime -1}_2,\al^{\prime -1}_3,$ and $\be^{\prime -1}_3$; note that the  $(2m_1+2m_2+2m_3+6)$th letter is always $\bar{e}_{n_3-2\ell+2}$. Specifically, $g$  maps the $i$th leftmost letter of the left-hand side of (\ref{equality}) to the $i$th leftmost letter of the right-hand side. 

First, suppose $g$ has no fixed points ($\mathfrak i$ such that $g(\mathfrak i)=\ii$). Then, use an algorithm similar to the one used in Section~\ref{sec:thm1} to conclude that there are $\le X/2$ degrees of freedom for the $G_1$-letters of $ABDEF$, and likewise for the $G_2$-letters. So, $W$ must be one of only $O\left(\goc^{X/2} \gtc^{X/2} \right)$ choices (for each choice of $\ell$ and $p'$).

Now, suppose that there exists an $\mathfrak i$ such that $g(\mathfrak i)=\ii$. Such fixed points $\ii$ must be letters of $A$, $B$, or $E$. We first consider the case that all the fixed points are letters of only one of $A$, $B$, and $E$. In this case, we consider the following subcases for $W$, $p$, and $p'$: 

\emph{Case 1}. Suppose that the fixed points are letters of $E$. Then, all of the fixed points must be in one of $w_2$ and $w_3$; they cannot be in $w_1$ since this would mean that $w_1$ contains $e_1$, but $e_1$ is necessarily located at different positions in the left-hand side and right-hand side of (\ref{equality}). If all the fixed points are in $w_2$, then we require that 
\[
1+(2m_3+1)+1=(2n_3-4\ell+3)-\left((2m_2+1)+1+(2m_3+1)+1\right),
\]
 in order for the first letter of $w_2$ to be at the same position in both the left-hand and right-hand side. Hence, we have $m_2+2m_3=n_3-2\ell-2$, which means there are $\le (n_3-2\ell-2)/2$ choices for $p'$ parametrized by $m_3 \le(n_3-2\ell-2)/2$ . For each such choice of $p'$, there are 
 \[
 2n_3-4\ell+3 -2\left((2m_3+1)+1+1\right)=2n_3-4\ell-4m_3-3
 \]
fixed letters excluding the first $2m_3+3$ and the last $2m_3+3$ letters of $E$, with $n_3-2\ell-2m_3-1$ of them in $G_1$ and $n_3-2\ell-2m_3-2$ in $G_2$. There are $\le \left(X-(n_3-2\ell-2m_3-1)\right)/2$ degrees of freedom for the non-fixed letters in $G_1$ and $\le \left(X-(n_3-2\ell-2m_3-2)\right)/2$  degrees of freedom for the fixed letters in $G_2$. Thus, counting across all choices of values for the letters, $p$, $p'$, and $\ell$,  we have that the number of additional Wicks commutators arising from this case is bounded from above by
\begin{align*}
\sum_{n_1=0}^{X-3} \sum_{n_2=0}^{X-3-n_1} \sum_{\ell=1}^{\floor{\frac{X-2-n_1-n_2}{2}}} \sum_{m_3=0}^{\floor{\frac{n_3-2\ell-2 }{2}}}  \goc^{\frac{X+n_3-2\ell-2m_3-1}{2}}&\gtc^{\frac{X+n_3-2\ell-2m_3-2}{2}} 
\\
&\ll \goc^X\gtc^X,
\end{align*}
which is dominated by our error term.

 Next, suppose the fixed letters are in $w_3$.  Then, it is necessary that $g(\bar{e}_1)=\be^{\prime -1}_3$ and $w_3$ is given by $E$ with the first and last letters omitted. Thus, we have that $m_3=n_3-2\ell$, so the number of possible choices for $p'$ is at most the number of partitions of $X-3-n_3+2\ell$ into two nontrivial parts, which is $X-3-n_3+2\ell$. For each choice of $p'$, there are $n_3-2\ell+1$ fixed letters in $G_1$ and $n_3-2\ell$ fixed letters in $G_2$. Additionally, there are $\le \left(X-(n_3-2\ell+1)\right)/2$  degrees of freedom for the non-fixed letters in $G_1$ and $\le \left(X-(n_3-2\ell)\right)/2$  degrees of freedom for the  non-fixed letters in $G_2$. Counting across all choices of values for the letters, $p$, $p'$, and $\ell$,  we have that the number of additional Wicks commutators arising from this case is bounded from above by
\begin{align*}
\sum_{n_1=0}^{X-3} \sum_{n_2=0}^{X-3-n_1} \sum_{\ell=1}^{\floor{\frac{X-2-n_1-n_2}{2}}} (X-3-n_3+2\ell)  \goc^{\frac{X-n_3+2\ell-1}{2}} &\gtc^{\frac{X-n_3+2\ell}{2}} \\
 &\ll \goc^X\gtc^X,
\end{align*}
which is dominated by our error term.

\emph{Case 2}.  Suppose that the fixed points are letters of $B$. Then, all of the fixed points must be in one of $w_1$, $w_2$, and $w_3$. First, suppose they are in $w_1$. This requires that $w_1=D^{-1}\be_3 A\al_1 B \be_1 V$, where $V$ is the  left subword of $D E$ having length $2n_1+2\ell+1$ (the length of $A\al_2^{-1}F^{-1}$). All letters of $B$ are thus included in $w_1$.   We have 
\begin{align*}
2m_1+1=(2\ell-1)+1+(2n_1+1)+1+(2n_2+1)&+1+(2n_1+2\ell+1)\\&=4n_1+2n_2+4\ell+5,
\end{align*}
so $m_1 = 2n_1+n_2+2\ell+2$. Thus, the number of possible choices for $p'$ is at most the number of partitions of $X-3-2n_1-n_2-2\ell-2$ into two nontrivial parts, which is $X-5-2n_1-n_2-2\ell \le X-3-n_1-n_2-2\ell$ (the latter is guaranteed to be nonnegative for any choice of $p$). The fixed letters are precisely the letters of $B$, so for each choice of $p'$, we have $n_2+1$ fixed $G_1$-letters, $n_2$ fixed $G_2$-letters, $\le (X-n_2-1)/2$  degrees of freedom for the non-fixed $G_1$-letters, and $\le (X-n_2)/2$  degrees of freedom for the non-fixed $G_2$-letters. Counting across all choices of values for the letters, $p$, $p'$, and $\ell$,  we have that the number of additional Wicks commutators arising from this case is bounded from above by
\begin{align*}
\sum_{n_1=0}^{X-3} \sum_{n_2=0}^{X-3-n_1}  \sum_{\ell=1}^{\floor{\frac{X-2-n_1-n_2}{2}}} (X-3-n_1-n_2-2\ell) \goc^{\frac{X+n_2+1}{2}}& \gtc^{\frac{X+n_2}{2}} \\&\ll \goc^X\gtc^X,
\end{align*}
which is dominated by our error term.

Next, suppose that the fixed points are in $w_2$. Then, we require that the difference between the lengths of  $EF\al_3^{-1}$ (length $2n_3-2\ell+3$) and $\be_3^{-1}A\al_2^{-1}F^{-1}$ (length $2n_1+2\ell+2$) is the same as that between $\be^{\prime-1}_3w_3\al_3^{\prime-1}$ (length $2m_3+3$) and $\be^{\prime-1}_2w_1$ (length $2m_1+2$). Furthermore, the number of (fixed) letters of $B$ in $w_2$ is $2n_2+1$ if 
\[
j \defeq \frac{2n_1+2\ell+2-(2m_1+2)}{2} = \frac{2n_3-2\ell+3-(2m_3+3)}{2}
\] 
is negative and $2(n_2-j)+1$ if $j \ge 0$. First, suppose that $j \ge 0$. In this  case, $p'$ is determined by the choice of $j \le n_2/2$, for which there are $n_2-2j+1$ fixed $G_1$-letters and $n_2-2j$ fixed $G_2$-letters of $B$. There are $\le \left(X-(n_2-2j+1)\right)/2$  degrees of freedom for the non-fixed letters of $G_1$ and $\le (X-n_2+2j)/2$ degrees of freedom for the non-fixed letters of $G_2$, so overall, we can count across all choices of values for the letters, $p$, $j$, and $\ell$ to get that the number of additional Wicks commutators arising from this case is bounded from above by
\begin{align*}
\sum_{n_1=0}^{X-3} \sum_{n_3=0}^{X-3-n_1}  \sum_{\ell=1}^{\floor{\frac{n_3}{2}}} \sum_{j=0}^{\floor{\frac{X-3-n_1-n_3}{2}}} \goc^{\frac{X+n_2-2j+1}{2}} &\gtc^{\frac{X+n_2-2j}{2}} \\
&\ll\goc^X\gtc^X,
\end{align*}
which is dominated by our error term.

Now, suppose that $j <0$. In this case, $-j=m_1-n_1-\ell=m_3-n_3+\ell$ is a positive integer satisfying $(n_1+\ell-j) +(n_3-\ell-j) = m_1+m_3 \le X-3$, which gives the condition $-j \le (X-3-n_1-n_3)/2 = n_2/2$. Note that $p'$ is determined by the choice of $-j$, for which there are $n_2+1$ fixed $G_1$-letters and $n_2$ fixed $G_2$-letters, all of which are in $B$. The non-fixed letters in $G_1$ have $\le (X-n_2-1)/2$ degrees of freedom and those in $G_2$ have $\le (X-n_2)/2$ degrees of freedom. Overall, we can count across all choices of values for the letters, $p$, $-j$, and $\ell$ to get   that the number of additional Wicks commutators arising from this case is bounded from above by
\begin{align*}
\sum_{n_1=0}^{X-3} \sum_{n_2=0}^{X-3-n_1}  \sum_{\ell=1}^{\floor{\frac{X-2-n_1-n_2}{2}}} \sum_{-j=1}^{\floor{\frac{n_2}{2}}}    \goc^{\frac{X+n_2-2j+1}{2}} &\gtc^{\frac{X+n_2-2j}{2}} \\\ll &\goc^X\gtc^X,
\end{align*}
which is dominated by our error term.

Finally, suppose that the fixed points are in $w_3$.  This requires that $w_3$ is given by $V B \be_1 DE$ with the last letter omitted, where $V$ is the right subword of $D^{-1}\be_3 A\al_1$ having length $(2n_3-4\ell+3)+(2\ell-1)=2n_3-2\ell+2$, one less than that of $EF\al_3^{-1}$. All letters of $B$ are thus included in $w_3$. We have that
 \begin{align*}
2m_3+1=(2n_3-2\ell+2)+(2n_2+1)+1+(2\ell-1)&+(2n_3-4\ell+2)\\&=2n_2+4n_3 - 4\ell+5,
\end{align*}
so $m_3=n_2+2n_3-2\ell+2$. Thus, the number of possible choices for $p'$ is at most the number of partitions of $X-5-n_2-2n_3+2\ell$ into two nontrivial parts, which is $X-5-n_2-2n_3+2\ell \le X-3-n_2+2\ell$ (the latter is guaranteed to be nonnegative for any choice of $p$). The fixed letters are precisely the letters of $B$, so there are $n_3+1$ fixed letters of $G_1$ and $n_3$ fixed letters of $G_2$, with the non-fixed $G_1$-letters having $\le (X-n_2-1)/2$ degrees of freedom and the non-fixed $G_2$-letters having $\le (X-n_2)/2$ degrees of freedom. Counting across all choices of values for the letters, $p$, $p'$, and $\ell$,  we have that the number of additional Wicks commutators arising from this case is bounded from above by
 For each choice of $p'$,
\begin{align*}
 \sum_{n_1=0}^{X-3} \sum_{n_2=0}^{X-3-n_1}  \sum_{\ell=1}^{\floor{\frac{X-2-n_1-n_2}{2}}} (X-3-n_2+2\ell)  \goc^{\frac{X+n_2+1}{2}} &\gtc^{\frac{X+n_2}{2}} \\ \ll &\goc^X\gtc^X,
\end{align*}
which is dominated by our error term.

\emph{Case 3}. Suppose that the fixed points are letters of $A$. Then, all of the fixed points must be in one of $w_1$ and $w_2$; they cannot be in $w_3$, since then there would be more than $2\ell$ letters right of $A$. First, suppose they are in $w_1$.  This requires that $w_1=D^{-1}\be_3^{-1}AV$, where $V$ is the left subword of $\al_1^{-1} B \be_1^{-1} DE$ having length $2\ell$, the length of $\al_2^{-1}F$. The fixed letters are then precisely the letters of $A$, so there are $n_1+1$ fixed letters of $G_1$ and $n_1$ fixed letters of $G_2$. The non-fixed $G_1$-letters have $\le (X-n_1-1)/2$ degrees of freedom and the non-fixed $G_2$-letters have $\le (X-n_1)/2$  degrees of freedom.The number of possible choices for $p'$ is at most the number of partitions of $X-3-m_1=X-3-(2\ell+2n_1+1+2\ell)=X-4-2\ell -n_1$ into two nontrivial parts, which is $\le X-3-2\ell-n_1$ (the latter is guaranteed to be nonnegative for any choice of $p$).   Counting across all choices of values for the letters, $p$, $p'$, and $\ell$,  we have that the number of additional Wicks commutators arising from this case is bounded from above by
\begin{align*}
 \sum_{n_1=0}^{X-3} \sum_{n_2=0}^{X-3-n_1}  \sum_{\ell=1}^{\floor{\frac{X-2-n_1-n_2}{2}}} (X-3-2\ell-n_1)  \goc^{\frac{X+n_1+1}{2}} &\gtc^{\frac{X+n_1}{2}} \\ \ll &\goc^X\gtc^X,
\end{align*}
which is dominated by our error term.

Finally, suppose the fixed points are in $w_2$. Since the length difference  $2m_1-2\ell+2$ between the lengths of $D^{-1}$ and $w_1$ must be equal to the length difference $2n_1+2\ell-2m_1-2m_2-2$ between the lengths of $D^{-1} \be_3 A$ and $w_1 \al'_1 w_2$, we require that $m_2=n_1+2\ell-2m_1-2$. Note then that $p'$ is parametrized by $m_1 \le (n_1+2\ell-2)/2$. For a given choice of $p'$, if $m_1 \le \ell$, then we have $n_1+1$ fixed letters of $G_1$ and $n_1$ fixed letters of $G_2$, with $\le (X-n_1-1)/2$ degrees of freedom for the non-fixed $G_1$-letters and $\le (X-n_1)/2$ degrees of freedom for the non-fixed $G_2$-letters. This gives  $\le (X+n_1+1)/2 \le (X+n_1-m_1+\ell+1)/2$ overall degrees of freedom for the $G_1$-letters, as well as  $\le (X+n_1-m_1+\ell)/2$ ones for the $G_2$-letters. If $m_1 > \ell$, then we have $2n_1-2(m_1-\ell)+1$ fixed letters of $G_1$ and $2n_1-2(m_1-\ell)$ fixed letters of $G_2$, so analogously we have $\le (X+n_1-m_1+\ell+1)/2$ overall degrees of freedom for the $G_1$-letters, as well as $\le (X+n_1-m_1+\ell)/2$ ones for for the $G_2$-letters. Thus, counting across all choices of values for the letters, $p$, $p'$, and $\ell$,  we have that the number of additional Wicks commutators arising from this case is bounded from above by
\begin{align*}
 \sum_{n_1=0}^{X-3} \sum_{n_2=0}^{X-3-n_1}  \sum_{\ell=1}^{\floor{\frac{X-2-n_1-n_2}{2}}} \sum_{m_1=0}^{\floor{\frac{n_1+2\ell-2}{2}}} \goc^{\frac{X+n_1-m_1+\ell+1}{2}} &\gtc^{\frac{X+n_1-m_1+\ell}{2}} \\ \ll &\goc^X\gtc^X,
\end{align*}
which is dominated by our error term.
\iffalse
 $n_1-(\ell-m_1)$  fixed letters if $m_1 \le \ell$

and for each choice of $p'$, we have $n_1-(\ell-m_1)$ degrees of freedom for the fixed letters and $\le (X-(n_1-(\ell-m_1)))/2$, for the non-fixed letters.  \fi

Next, we suppose that the fixed letters of $g$ are in two of the three subwords $A$, $B$, and $E$. Consider the following subcases:

\emph{Case 1}. Suppose the fixed letters of $g$ are in $A$ and $B$. It is necessary that the fixed letters of $A$ and those of $B$ are in $w_i$ and $w_j$, respectively, such that $i < j$; otherwise,  the fixed letters of $A$ would come before the fixed letters of $B$, a contradiction. First, suppose that the fixed letters of $B$ are in $w_2$, which implies that the fixed letters of $A$ are in $w_1$. Then, we require that $w_1=D^{-1}\be_3AV$, where $V$ is the left subword of $\al_1B\be_1DE$ having length $2\ell$. This gives $2m_1+1 =2n_1+ 4\ell +1$, i.e, $m_1 = n_1+2\ell$. Furthermore, since we have fixed letters of $B$, we require that $V$ does not include all of $B$, i.e., $2n_2+2>2\ell$, or equivalently, $n_2 \ge \ell$.  Next, for the fixed letters of $B$ to match in position, we require that $w_2$ ends at the letter $b_{2n_2-2\ell+1}$, which gives us  $2m_2+1=(2n_2-2\ell+1)-(2\ell+1)+1=2n_2-4\ell+1$, i.e., $m_2 = n_2-2\ell \ge 0$. Then, $m_3$ is automatically determined, and in particular, $w_3$ is the subword of $b_{n_2-\ell+2} \cdots b_{n_2+1} \be_1 DE$ omitting the rightmost letter. For this  $p'$ corresponding to $p$, we have $2n_1+1$ fixed letters of $A$ ($n_1+1$ letters of $G_1$ and $n_1$ letters of $G_2$) and $2n_2-4\ell+1$ fixed letters of $B$ ($n_2-2\ell+1$ letters of $G_1$ and $n_2-2\ell$ letters of $G_2$). Next, we bound the degrees of freedom of the non-fixed letters. Note that $g$ maps the letters of $EF\al_3^{-1}b_1 \cdots b_{\ell-1}\bar{b}_{\ell-1}$ to those of $\be_3^{\prime -1} b_{n_2-\ell+2} \cdots b_{n_2+1} \be_1 DE$ in order, but  $g$ also maps $f_{\ell}^{-1}, \ldots, \bar{f}_2^{-1} f_1^{-1}$ to $b_1 ,\ldots, \bar{b}_{\ell-1}, b_{\ell}$ and $ b_{n_2-\ell+2}, \ldots, b_{n_2+1}$, to $d_{\ell}^{-1},\ldots, d_1^{-1}$. Thus, arguing inductively by translation, we see that that choosing the letters of $F$ determines the letters of $E$, and thus also determines those of $D$, thereby determining all non-fixed letters (while not caring about the constant number of $\al_i$ and $\be_i$ letters). There are $\ell$ $G_1$-letters and $\ell-1$ $G_2$-letters in $F$. Counting across all choices of values for the letters, $p$, and $\ell$, we have that the number of additional Wicks commutators arising from this case is bounded from above by
\begin{align*}
\sum_{n_1=0}^{X-3} \sum_{n_2=0}^{X-3-n_1} \sum_{\ell=1}^{\floor{\frac{X-2-n_1-n_2}{2}}} \goc^{(n_1+1)+(n_2-2\ell+1)+\ell} &\gtc^{n_1+(n_2-2\ell)+(\ell-1)} \\ \ll &X\goc^X\gtc^X,
\end{align*}
which is dominated by our error term.

Next, suppose that the fixed letters of $B$ are in $w_3$. Then, we require that $w_3=VB\be_1 D E$, where $V$ is the right subword of $D^{-1}\be_3 A \al_1$ having length $2n_3-2\ell+2$, the length of $EF$.  Furthermore, since we have fixed letters of $A$, we require that $V$ does not include all of $A$, i.e., $2n_3-2\ell+2 < 2n_1+2$, or in other words, $n_3 < n_1+\ell$.  Next, for the fixed letters of $A$ to match in position, we require that they are in $w_2$, and specifically that $w_2=a_{n_3-\ell+1}\cdots a_{n_1-n_3+\ell}$. This gives us $2m_2+1=2n_1-4n_3+4\ell-1$, and thus $m_2 =n_1-2n_3+2\ell-1 > 0$. It follows that $w_1=D^{-1}\be_3 a_{1} \cdots a_{n_3-\ell}$. In particular, $m_1$ is automatically determined, and  for this  $p'$ corresponding to $p$, we have $n_2+1$ fixed $G_1$-letters and $n_2$ fixed $G_2$-letters of $B$, as well as $n_1-2n_3+2\ell$ fixed $G_1$-letters and $n_1-2n_3+2\ell-1$ fixed $G_2$-letters of $A$. Next, we bound the degrees of freedom of the non-fixed letters. Note that $g$ maps the letters of $a_{n_1-n_3+\ell+1}\cdots a_{n_1} \al_2^{-1} F^{-1}$ to those of $D^{-1}\be_3 a_{1} \cdots a_{n_3-\ell}$ in order. However, $g$ maps $d_1, \ldots, d_{\ell}, \bar{e}_1, \ldots, e_{n_3-2\ell+1}$ to $a_1, \bar{a}_1, \ldots, a_{n_3-\ell+1}, \bar{a}_{n_3-\ell+1}$. Also, $g$ maps $a_{n_1-n_3+\ell+1}, \cdots,  a_{n_1+1}$, to $e_2, \ldots, \bar{e}_{n_3-2\ell+2}, f_1, \ldots, f_{\ell}$. Thus, arguing inductively by translation, we see that that choosing the letters of $F$ determines the letters of $E$, and thus also determines those of $D$, thereby determining all non-fixed letters (while not caring about the constant number of $\al_i$ and $\be_i$ letters). Counting across all choices of values for the letters, $p$, and $\ell$, we have that the number of additional Wicks commutators arising from this case is bounded from above by
\begin{align*}
\sum_{n_1=0}^{X-3} \sum_{n_2=0}^{X-3-n_1} \sum_{\ell=1}^{\floor{\frac{X-2-n_1-n_2}{2}}} \goc^{(n_1-2n_3+2\ell)+(n_2+1)+\ell} &\gtc^{(n_1-2n_3+2\ell-1)+n_2+(\ell-1)}\\ \ll &X\goc^X\gtc^X,
\end{align*}
which is dominated by our error term.

\emph{Case 2}. Suppose the fixed letters of $g$ are in $B$ and $E$. Similarly to before, it is necessary that the fixed letters of $B$ and those of $E$ are in $w_i$ and $w_j$, respectively, such that $i < j$.  First, suppose that the fixed letters of $B$ are in $w_1$. Then, we require that $w_1=D^{-1}\be_3A\al_1BV$, where $V$ is the left subword of $\be_1DE$ having length $2n_1+2\ell+2$, the length of $\be_3^{-1}A\al_2^{-1}F^{-1}$.  Then, $w_2$ must start with $e_{n_1+3}$, which means in order to have the letters of $E$ match, we must have $2m_3+3=2n_1+3$, i.e., $m_3=n_1$. Since $(2n_1+2)+1+(2m_2+1)+1+(2m_3+1)+1 =2n_3-4\ell+3$ (counting the letters of $E$ in two ways), we have $m_2=n_3-2n_1-2\ell-2>0$. There are $n_3-2n_1-2\ell-1$ fixed $G_1$-letters and $n_3-2n_1-2\ell-2$ fixed $G_2$-letters of $E$ in $w_2$. Also, there are $n_2+1$ fixed $G_1$-letters  and $n_2$ fixed $G_2$-letters of $B$. Next, we bound the degrees of freedom of the non-fixed letters. Note that $g$ maps the letters of $A\al_2^{-1}F^{-1}$ to those of  $D \bar{e}_1 \cdots \bar{e}_{n_1+1}e_{n_1+2}$ in order. 
Likewise, by observing the left end of $w_1$, we see that $g$ maps the letters of $e_{n_3-n_1-2\ell+2} \cdots \bar{e}_{n_3-2\ell+2} F$ to those of $D^{-1}\be_3 A$ in order. However, we also have that $g$ maps the letters of $\bar{e}_1 \cdots \bar{e}_{n_1+1}e_{n_1+2}$ to those of $\be_3^{\prime -1} e_{n_3-n_1-2\ell+2} \cdots \bar{e}_{n_3-2\ell+1} e_{n_3-2\ell+2}$ in order, which overall gives us the ordered equality (by $g$) of letters
\begin{align*}
A\al_2^{-1}F^{-1}\bar{e}_{n_3-2\ell+2}F& =D \bar{e}_1 \cdots e_{n_1+2} \bar{e}_{n_3-2\ell+2}F
\\ & = D\be_3^{\prime -1} e_{n_3-n_1-2\ell+2} \cdots  e_{n_3-2\ell+2} \bar{e}_{n_3-2\ell+2} F=D\be_{3}^{\prime -1} D^{-1}\be_3A.
\end{align*}
Thus, arguing inductively by translation, we see that that choosing the letters of $F$ determines the letters of $A$, and thus also determines those of $D$, thereby determining all non-fixed letters (while not caring about the constant number of exceptional letters). Counting across all choices of values for the letters, $p$, and $\ell$, we have that the number of additional Wicks commutators arising from this case is bounded from above by
\begin{align*}
\sum_{n_1=0}^{X-3} \sum_{n_2=0}^{X-3-n_1} \sum_{\ell=1}^{\floor{\frac{X-2-n_1-n_2}{2}}}  \goc^{(n_2+1)+(n_3-2n_1-2\ell-1)+\ell} &\gtc^{n_2+(n_3-2n_1-2\ell-2)+(\ell-1)}\\ \ll &X\goc^X\gtc^X,
\end{align*}
which is dominated by our error term.

Next, suppose that the fixed  letters of $B$ are in $w_2$, which implies the fixed letters of $E$ are in $w_3$. Then, we require that $w_3=e_2\cdots e_{n_3-2\ell+2}$. Furthermore, in order for the letters of $B$ to match in position, we must have that $w_2=V\al_1 B \be_1 D$, where $V$ is the right subword of $D^{-1}\be_3A$ having length $2\ell-1$.  We thus have $n_3-2\ell+1$ fixed $G_1$-letters and $n_3-2\ell$ fixed $G_2$-letters in $E$, as well as $n_2+1$ fixed $G_1$-letters and $n_2$ fixed $G_2$-letters in $B$. Now, we bound the degrees of freedom of the non-fixed letters. First, suppose that $n_1>\ell$. Note then that $g$ maps the letters of $F$ to those of $a_{n_1-\ell+2} \cdots a_{n_1+1}$ in order, and likewise maps the letters of  $A\al_2^{-1}F^{-1}$ to those of $D \be_2^{\prime -1} D^{-1} \be_3 a_1 \cdots a_{n_1-\ell+1}$ in order. Thus, we have the ordered equality  (by $g$)  of letters 
\begin{align*}
a_1 \cdots a_{n_1-\ell+1}  \bar{a}_{n_1-\ell+1} F \al_2^{-1}F^{-1} &=  a_1 \cdots a_{n_1-\ell+1}  \bar{a}_{n_1-\ell+1} (a_{n_1-\ell+2} \cdots a_{n_1+1}) \al_2^{-1}F^{-1}\\&=  D \be_2^{\prime -1}D^{-1}\be_3 a_1 \cdots a_{n_1-\ell-1}.
\end{align*}
 Thus, arguing inductively by translation, we see that that choosing the letters of $F$ determines the letters of $a_1 \cdots a_{n_1-\ell-1}$, and thus also determines those of the rest of $A$ and of $D$, thereby determining all non-fixed letters (while not caring about the constant number of exceptional letters). In the  other case of  $n_1 \le \ell$, the notation above for $a_1 \cdots a_{n_1-\ell-1}$ becomes inviable, but nevertheless we can use a similar argument as above to conclude that choosing the letters of $F$ determines all the non-fixed letters. Counting across all choices of values for the letters, $p$, and $\ell$, we have that the number of additional Wicks commutators arising from this case is bounded from above by
\begin{align*}
\sum_{n_1=0}^{X-3} \sum_{n_2=0}^{X-3-n_1} \sum_{\ell=1}^{\floor{\frac{X-2-n_1-n_2}{2}}} \goc^{(n_2+1)+(n_3-2\ell+1)+\ell} &\gtc^{n_2+(n_3-2\ell)+(\ell-1)}\\ \ll &X\goc^X\gtc^X,
\end{align*}
which is dominated by our error term.

\emph{Case 3}. Finally, suppose the fixed letters of $g$ are in $A$ and $E$. Similarly to before, it is necessary that the fixed letters of $A$ and those of $E$ are in $w_i$ and $w_j$, respectively, such that $i < j$. First, suppose that the fixed letters of $E$ are in $w_3$. Then, we require that $w_3=e_2\cdots e_{n_3-2\ell+2}$. Furthermore, in order for the letters of $A$ to match in position, we need that the fixed letters of $A$ are contained in $w_1$, and in particular, that $w_1=D^{-1}\be_3A \al_1 V$, where $V$ is the left subword of $B\be_1DE$ having length $2\ell-1$. We thus have $n_3-2\ell+1$ fixed $G_1$-letters and $n_3-2\ell$ fixed $G_2$-letters in $E$, as well as $n_1+1$ fixed $G_1$-letters and $n_1$ fixed $G_2$-letters in $A$.  Now, we bound the degrees of freedom of the non-fixed letters. First, suppose that $n_2>\ell$.
 Note then that $g$ maps the letters of $F^{-1}$ to those of $b_1 \cdots b_{\ell}$ in order, and likewise maps the letters of $F\al_3^{-1} B$ to those of $b_{\ell+1} \cdots b_{n_2+1} \be_1 D \be_2^{\prime -1} D^{-1}$. Thus, we have the ordered equality  (by $g$)  of letters 
 \begin{align*}
 F \al_3^{-1} F^{-1}  \bar{b}_{\ell} b_{\ell+1} \cdots b_{n_2+1}  &= F \al_3^{\prime -1} (b_1 \cdots b_{\ell} )\bar{b}_{\ell} b_{\ell+1} \cdots b_{n_2+1}  \\&=b_{\ell+1} \cdots b_{n_2+1} \be_1 D \be_2^{\prime -1} D^{-1}
 \end{align*}
Arguing inductively by translation, we see that that choosing the letters of $F$ determines the letters of $B$ of $D$, thereby determining all non-fixed letters. In the  other case of  $n_2 \le \ell$, the notation above for $b_1 \cdots b_{\ell}$ becomes inviable, but nevertheless we can use a similar argument as above to conclude that $F$ determines all the non-fixed letters. Counting across all choices of values for the letters, $p$, and $\ell$, we have that the number of additional Wicks commutators arising from this case is bounded from above by
\begin{align*}
\sum_{n_1=0}^{X-3} \sum_{n_2=0}^{X-3-n_1} \sum_{\ell=1}^{\floor{\frac{X-2-n_1-n_2}{2}}}\goc^{(n_1+1)+(n_3-2\ell+1)+\ell} &\gtc^{n_1+(n_3-2\ell)+(\ell-1)}\\ \ll &X\goc^X\gtc^X,
\end{align*}
which is dominated by our error term.

Next, suppose that the fixed letters of $E$ are in $w_2$. Then, the fixed letters of $A$ are contained in $w_1$, which requires that $w_1=D^{-1}\be_3A \al_1 V$, where $V$ is the left subword of $B\be_1DE$ having length $2\ell-1$. But then $w_2$ must start on a letter not in $E$, which makes it impossible for the letters of $E$ to match in position. 

Finally, if there are fixed letters in $A, B,$ and $C$, then it is necessarily that $\ell = 0$ and $p=p'$, which does not need to be considered.

If $\ell >(n_3+1)/2$, then we can think of our commutator as a cyclic conjugate of $C^{-1}ABCA^{-1}B^{-1}$ such that the letters are moved from left to right. A symmetric argument like above gives us the conclusion that $W'=D^{-1}\be_3 A\al_1 B\be_1 D E F\al_2 A^{-1}\be_2 B^{-1} \al_3 F^{-1}E^{-1}$ is on average not a generic Wicks commutator of the form $w_1 \al'_1 w_2 \be'_1 w_3 \al'_2 w_1^{-1} \be'_2 w_2^{-1} \al'_3 w_3^{-1} \be'_3$. Note that this entire argument can then be repeated \emph{mutatis mutandis} to show that $W'$ is on average also not a Wicks commutator of the other generic form $(9)$ or of the other possible forms $(3)-(7)$, since the only difference from the previous case is a constant number of exceptional letters and possibly setting one or more of the subwords $w_1, w_2,$ and $w_3$ to be trivial, which overall can only affect error bounds by at worst a multiplicative constant. Likewise, a similar argument shows that $W'$ is also not a Wicks commutator for the case of $\ell=0$, again since the only difference from the previous case is a constant number of exceptional letters. Finally, a similar argument shows that even if $W$ is taken to be of the generic form $(9)$ rather than $(8)$, $W$ is on average only decomposable as a Wicks commutator in one prescribed way (i.e., with respect to a unique partition of $X-3$), and any cyclic conjugate of $W$ is on average not a Wicks commutator. 

We have thus shown that the number of conjugacy classes of commutators with length $4X$ is given by
\begin{align*}
\frac{1}{6} \bigg( \frac{1}{2}\big(\goc( |G_2|&-2 )^2 +( |G_1|-2 )^2 \gtc \big) \\
& \cdot X^2  \goc^{X-1} \gtc^{X-1} + O\left( X\goc^{X} \gtc^{X}  \right)\bigg)\\
= \frac{1}{12}\bigg(&\goc( |G_2|-2 )^2 +( |G_1|-2 )^2 \gtc \bigg) \\& \cdot X^2  \goc^{X-1} \gtc^{X-1} +O\left( X\goc^{X} \gtc^{X}  \right),
\end{align*}
as needed. 

\section{Concluding Remarks}\label{sec:conclusion}

There are a number of directions in which  Theorems~\ref{thm1} and \ref{thm2} can be generalized. First, one can ask: how many conjugacy classes of commutators with word length $k$ are in an arbitrary finitely-generated free product $G = G_1 * \cdots *G_n$, with each $G_i$ nontrivial and having symmetric generating set $\cS_i=\{g_{1}^{(i)}. \ldots, g_{m_i}^{(i)},(g_{1}^{(i)})^{-1}. \ldots, (g_{m_i}^{(i)})^{-1}\}$? While one can define the word length in this context to be with respect to an arbitrary generating set $\cS$, a natural notion of length to use in this setting would be with respect to the symmetric generating set
\[
\cS\defeq \bigcup_{i=1}^{n} \cS_i.
\]
We note that in the case that $G_i$ is finite for all $1 \le i \le n$, we can take $\cS_i = G_i \setminus\{1\}$, for which $\cS$ is consistent with our choice of the set $\cS$ of generators for $G_1 * G_2$ in the statement of Theorem~\ref{thm2}. Counting conjugacy classes of commutators by word length for groups in this more general  form would have analogous geometric consequences as those discussed in Corollary~\ref{cor1}. For example,  Hecke Fuchsian groups $H(\ld)$ for $\ld \ge 2$ have the presentation~\cite{nexthecke}
\[
\langle S,R_\ld : S^2=I \rangle \cong \ZZ/2\ZZ *\ZZ,
\]
and a desire for geometric corollaries akin to those discussed in Section~\ref{sec:intro} motivates a result analogous to Theorem~\ref{thm1} and  Theorem~\ref{thm2} in the case of the free product of a nontrivial finite group and an infinite cyclic group. 

To describe a second potential direction for generalization, let the \emph{$n$-commutators} of a given group be defined by  the elements with trivial abelianization and commutator length $n$. A second direction for generalizing Theorems~\ref{thm1} and \ref{thm2} is to, for any $n$, count the number of conjugacy classes of $n$-commutators with word length $k$ in a group $G$ taken to be either the free group $F_r$ or a finitely generated free product. This is natural to ask, given that Culler~\cite{culler} has classified the possible forms of $n$-commutators for a free group and Vdovina~\cite{vdovina} has done this for an arbitrary free product. In fact, Culler has also classified the possible forms that a product of $n$ square elements can take for a free group, so an analogous question can be asked for the number of conjugacy classes comprised of products of $n$ square elements. If one obtained the asymptotic number of $n$-commutators with length $k$ (respectively, of $n$-square-element-products with length $k$), this would also give a corollary analogous to Corollary~\ref{cor1}. Specifically, given a connected CW-complex $X$ with fundamental group $G$, one would obtain the asymptotic number of free homotopy classes of loops $\gamma : S^1 \to X$ with length $k$ (in the generators of $\cS$) such that there exists a genus-$n$ orientable surface $Y$ with one boundary component (respectively,  a connected sum  $Y$ of $n$ real projective planes, with one boundary component) and a continuous map $f : Y \to X$ satisfying $f(\partial Y) = \Im \gamma$, but also that this statement does not hold when replacing $n$ with any $m \le n$. We expect our combinatorial method to also work in this generalized setting, as long as one has, for the given group, an explicit list of the possible Wicks forms of $n$-commutators (or of products of $n$ square elements).

One important question regarding the growth of groups is whether or not a given growth series $S(z) \defeq \sum_{X=0}^\infty a_X z^X$ is a rational function. For $S(z)$ to be rational, it is necessary that 
\begin{equation}\label{growth}
a_X \sim \sum_{i=1}^n c_i X^{m_i}\alpha_i^X
\end{equation}
for some complex numbers $\alpha_1,\ldots, \alpha_n,c_1,\ldots,c_n$ and nonnegative integers $m_1,\ldots, m_n$, as a function of $X$~\cite[p. 341]{rational}. Consequently, estimates on the number of conjugacy classes of commutator-subgroup elements in $F_r$ of length $k$ show that its growth series cannot be rational~\cite{rational1,rational2,rational3,sharp}. On the other hand, we have shown that the number of conjugacy classes of commutators of length $2X$ in $F_r$ (or those of length $4X$ in $G_1*G_2$) has an asymptotic that is of the desired form (\ref{growth}), which raises the question of whether or not its growth series is rational.

\section*{Acknowledgments}

 This work began at Princeton University as part of the author's senior thesis advised by Peter Sarnak, whom the author would like to thank for providing the number-theoretic motivation for this problem, invaluable discussions/references, and constant encouragement. 

The generalization of Theorem~\ref{thm2} to an arbitrary free product of two nontrivial finite groups was done at Harvard University and was supported by the National Science Foundation Graduate Research Fellowship Program (grant number DGE1745303). The author would like to thank Bena Tshishiku for reading over the manuscript and providing many valuable suggestions. He would also like to thank Aaron Calderon, Keith Conrad, Noam Elkies, Curt McMullen, Hector Pasten, Xiaoheng Wang, and Boyu Zhang for  very helpful discussions.

\section*{Acknowledgments}

\noindent  This work began at Princeton University as part of the author's senior thesis advised by Peter Sarnak, whom the author would like to thank for providing the number-theoretic motivation for this problem, invaluable discussions/references, and constant encouragement. 

The generalization of Theorem~\ref{thm2} to an arbitrary free product of two nontrivial finite groups was done at Harvard University and was supported by the National Science Foundation Graduate Research Fellowship Program (grant number DGE1745303). The author would like to thank Bena Tshishiku for reading over the manuscript and providing many valuable suggestions. He would also like to thank Aaron Calderon, Keith Conrad, Noam Elkies, Curt McMullen, Hector Pasten, Xiaoheng Wang, and Boyu Zhang for  very helpful discussions.

\bibliographystyle{amsplain}
\bibliography{main}

\end{document}